\documentclass[12pt]{amsart} 
\usepackage{amsfonts}
\usepackage{amssymb} 
\usepackage{graphicx}
\usepackage{color}
\usepackage[all]{xy}
\usepackage{lmodern}
\usepackage{pinlabel}
\usepackage{hyperref}
\hypersetup{backref,colorlinks=true,citecolor=blue,linkcolor=blue}

\newtheorem{theorem}{Theorem}[section]
\newtheorem{proposition}[theorem]{Proposition}
\newtheorem{lemma}[theorem]{Lemma}
\newtheorem{corollary}[theorem]{Corollary}
\theoremstyle{definition}
\newtheorem{definition}[theorem]{Definition}

\newtheorem{example}[theorem]{Example}

\newtheorem{remark}[theorem]{Remark}

\newtheorem{construction}[theorem]{Construction}

\renewenvironment{proof}{{\noindent\bf Proof.}}{\hfill $\Box$\par\vskip3mm}

\begin{document}

\def\cC{{\mathcal{C}}}
\def\cE{{\mathcal{E}}}
\def\cP{{\mathcal{P}}}
\def\cF{{\mathcal{F}}}
\def\cS{{\mathcal{S}}}
\def\cH{{\mathcal{H}}}
\def\cA{{\mathcal{A}}}
\def\cM{{\mathcal{M}}}
\def\cG{{\mathcal{G}}}
\def\cD{{\mathcal{D}}}

\def\vG{{\vec{{\mathcal{G}}}}}
\def\vGp{{\vec{{\mathcal{G}'}}}}

\title[Filtered screens and augmented Teichm\"uller space]
{Filtered screens and augmented Teichm\"uller space}

\author{Douglas J. LaFountain}
\address{Department of Mathematics \\ Western Illinois University}
\email{d-lafountain@wiu.edu}

\author{R.C. Penner}
\address{IHES\\ {\small and} Departments of Mathematics and Theoretical Physics\\ Caltech}
\email{rpenner@caltech.edu}


\thanks{Both authors
have been supported by the Center 
for Quantum Geometry of Moduli Spaces
funded by the Danish National Research Foundation.
It is furthermore a pleasure to thank Scott Wolpert for helpful input.}

\begin{abstract}  We study a new bordification of the decorated Teichm\"uller space for a multiply punctured surface $F$ by a space of {\it filtered screens} on the surface that arises from a natural elaboration of earlier
work of McShane-Penner.  We identify necessary and sufficient conditions for paths in this space of filtered screens to yield short curves having vanishing length in the underlying surface $F$.  As a result, an appropriate quotient of this space of filtered screens on $F$ yields a {\it decorated augmented Teichm\"uller space} which is shown to admit a  
CW decomposition that naturally projects to the augmented Teichm\"uller space by forgetting decorations and whose strata are indexed by a new object termed {\it partially oriented stratum graphs}.  
\end{abstract}

\maketitle

%*********************************************************
\section{Introduction}
\label{sec:intro}

Fatgraphs have proven to be useful tools for a number 
of enumerative problems related to the homology of
Riemann's moduli space \cite{AC,Harer86,HZ,Igusa1,Igusa2,Kontsevich,Mondello,Pert,Pen93b}.  A primary reason for this is that the decorated Teichm\"uller space of a punctured surface, where the decorations correspond to lengths of horocycles about each puncture, admits a mapping class group-invariant simplicial decomposition where simplices are indexed by isotopy classes of fatgraphs, with one barycentric coordinate per edge of the corresponding fatgraph in any given simplex.

A natural way to construct a mapping class group-invariant bordification of decorated Teichm\"uller space is to consider simplices corresponding to fatgraphs, and allow sub-fatgraphs of simplicial coordinates to vanish.  The results are then interpreted as nodal surfaces, where irreducible components are component sub-fatgraphs whose simplicial coordinates all vanish at the same rate.  The resulting bordification is a space of partially paired fatgraphs having punctures paired at nodes, and admits a CW decomposition where cells are products of simplices corresponding to irreducible component fatgraphs that necessarily have both decorated and undecorated punctures.

The idea of this combinatorial construction is not new, having appeared in some form in earlier published works of Kontsevich and Looijenga \cite{Kontsevich, [L]}.  However, what is new in the current work is that we relate the topology of this combinatorial space of partially paired fatgraphs directly to the topology of the well-studied augmented Teichm\"uller space, wherein short curves are allowed to have zero length.  Specifically, the above combinatorial construction begs a key question which has not been asked by previous authors:  In allowing sub-fatgraphs of simplicial coordinates to vanish, how does the resulting combinatorial path in the space of partially paired fatgraphs relate to the induced path of hyperbolic structures in augmented Teichm\"uller space? Specifically, are the relative boundaries of the vanishing sub-fatgraphs in actuality the short curves resulting in nodal double points in the augmented Teichm\"uller space, or are there other short curves which arise?  In fact we do show that these relative boundaries of vanishing sub-fatgraphs are precisely the short curves arising at limiting points in augmented Teichm\"uller space; this is the content of our first main Theorem \ref{thm:FSversionMcShanePenner}. The proof of this is non-trivial and is the main goal of our Section \ref{sec:Screen} which concerns {\it screens} and {\it filtered screens} on fatgraphs.  In brief, what is required is to relate the asymptotics of local simplicial coordinates to the asymptotics of global lambda lengths (or Penner coordinates) which relate directly to the hyperbolic lengths of embedded arcs in the punctured surface, and which can thus be used to identify precisely the short curves by applying previous work of McShane-Penner \cite{[MP]}. In this sense the current work is a sequel to, and completion of, the McShane-Penner paper

This leads to the second new aspect of our work in Section \ref{sec:DecaugTeich} which is to describe an appropriate quotient of the space of filtered screens which yields a {\it decorated augmented Teichm\"uller space}.  This space is shown to admit a mapping class group-invariant CW decomposition, the description of which is the content of our second main Theorem \ref{thm:maincelldecomp}.   An important question is whether the quotient of this space by the mapping class group has the same homotopy type as the Deligne-Mumford compactification of Riemann's moduli space, $\bar{\mathcal M}(F)$.  This topic has received much attention \cite{Gadgil}, and just recently a newly published manuscript of Z\'u\~niga presents a positive answer to the question of homotopy equivalence \cite{Zuniga}.  An open problem then remains as to determine how this quotient space could be potentially used to find explicit non-trivial cycles with which to broaden our understanding of the homology of $\bar{\mathcal M}(F)$.

The structure of the paper is as follows:  Section \ref{sec:background} serves to review needed aspects of decorated Teichm\"uller theory, in particular as it pertains to punctured surfaces where proper subsets of the punctures are decorated.  Section \ref{sec:Screen} introduces {\em filtered screens} on fatgraphs, and establishes the fact that filtered screens are necessary and sufficient to identify short curves with vanishing hyperbolic length for paths in decorated Teichm\"uller space.  It is thus filtered screens with which a bordification of decorated Teichm\"uller space, namely the {\em space of filtered screens} $\cF\cS(F)$, is constructed.  In Section \ref{sec:DecaugTeich}, an appropriate quotient of $\cF\cS(F)$ then yields the decorated augmented Teichm\"uller space $\widehat T(F)$, whose CW decomposition is succinctly described using the combinatorial objects {\it partially oriented stratum graphs}.

%--------------------------------------------------------

%*********************************************************
\section{Background and notation}
\label{sec:background}

This section recalls the decorated Teichm\"uller theory of punctured surfaces from \cite{[MP],[P]} which is systematically treated among other topics in the monograph \cite{[P3]}, to which we refer for further detail.

\subsection{Decorated Teichm\"uller Spaces}
The Teichm\"uller space 
$T(F)$ of $F=F_g^s$ is made up of all conjugacy classes of
discrete and faithful representations of $\pi_1(F)$ in $PSL_2(\mathbb{R})$ so that peripheral elements map to parabolics.  
The decorated Teichm\"uller space is the trivial bundle 
$$\widetilde{T}(F)=T(F) \times \mathbb{R}_{>0}^s\to T(F),$$ where the fiber $\mathbb R_{>0}^s$ over a point is interpreted geometrically as an $s$-tuple of hyperbolic lengths of not necessarily embedded horocycles, one about each puncture.  

\begin{definition}[Decorated and undecorated punctures, $\widetilde{T}_{\mathcal P}(F)$]
Consider a non-empty collection $\mathcal{P}$ of punctures on $F$, which we will refer to as {\em decorated punctures}, with an assignment of length of horocycle the associated {\em decorations}; any puncture not in $\mathcal{P}$ will be said to be {\em undecorated}.  The ${\mathcal P}$-{\em decorated Teichm\"uller space} is then defined to be
$\widetilde{T}_{\mathcal P}(F)=T(F) \times \mathbb{R}_{>0}^{{\mathcal P}}$.
\end{definition}

The union
$\widetilde T_*(F)=\sqcup_{{\mathcal P}\neq\emptyset}~ \widetilde{T}_{\mathcal P}(F)$ inherits a topology from
$T(F)\times {\mathbb R}_{\geq 0}^s$, where a horocycle is taken to be absent if it has length zero.
The mapping class group $MC(F)$ of isotopy classes of orientation-preserving homeomorphisms of $F$ acts on $T(F)$ by push forward of metric and on $\widetilde T(F)$ and $\widetilde T_*(F)$ by push forward of metric and decoration whereas only the subgroup of $MC(F)$ setwise fixing ${\mathcal P}$ acts on $\widetilde T_{\mathcal P}(F)$.

It is often convenient to projectivize decorations, for which we introduce the
following notation. For each $m \geq 0$, let $P\mathbb{R}_{\geq 0}^{m+1}$ denote  the standard closed $m$-simplex in $\mathbb{R}_{\geq 0}^{m+1}$ with barycentric coordinates $t_i\geq 0$, for $i=1,\ldots, m+1$, satisfying $\sum_{i=1}^{m+1} t_i = 1$.
Define {\em projectivized decorated Teichm\"{u}ller space}  $P\widetilde{T}(F) = \widetilde T(F)/{\mathbb R}_{>0}\approx T(F)\times P\mathbb{R}_{>0}^s$ and likewise for $P\widetilde{T}_{\mathcal P}(F) \approx T(F)\times P\mathbb{R}_{>0}^{\cP}$ and
$P\widetilde{T}_*(F) \approx T(F)\times P\mathbb{R}_{\geq0}^s$,
where in each case
we decorate punctures with a projectivized tuple of hyperbolic lengths of all horocycles summing to one.

\begin{definition}[Ideal arc, quasi cell decomposition]
An {\em ideal arc} $e$ in $F$ is (the isotopy class of) an arc embedded in $F$ with its endpoints at the punctures.  A 
{\it quasi cell decomposition} or q.c.d.\  of $F$ is a collection of ideal arcs pairwise disjointly embedded in $F$ except perhaps at their endpoints, no two of which are parallel,  so that each complementary component is either a polygon or an exactly once-punctured polygon.  A q.c.d.\ is said to be {\it based} at the collection of punctures arising as endpoints of its arcs, with these punctures regarded as the subset of decorated punctures $\mathcal{P}$.  
\end{definition}

A q.c.d.\ is in particular
an {\it ideal cell decomposition} or i.c.d.\ if each complementary region is unpunctured.  A maximal i.c.d.\  (and q.c.d.\ respectively) 
has only complementary triangles (as well as once-punctured monogons) and is 
called an {\it ideal triangulation} (and a {\it quasi triangulation}). 

\subsection{Coordinates and Parameters}
In this section we review the definitions of {\em lambda lengths, h-lengths} and {\em simplicial coordinates}.

\begin{definition}[Lambda lengths]
Associated to $\widetilde{\Gamma} \in \widetilde{T}(F)$ and an ideal arc $e$ in $F$, we may assign a positive real number called its {\em lambda length} or {\it Penner coordinate} as follows:  Straighten $e$ to a $\Gamma$-geodesic  connecting punctures in $F$, truncate this bi-infinite geodesic  to a finite geodesic arc using the horocycles from $\widetilde{\Gamma}$, let $\delta$ denote the signed hyperbolic length of this truncated geodesic (taken with positive sign if and only if the horocycles are disjoint), and set 
$$\lambda(e;\widetilde{\Gamma}) = \exp{\delta\over 2}.$$ When $\tilde\Gamma\in\widetilde T(F)$ is fixed or understood, we may write simply $\lambda(e)$.
\end{definition}

\begin{theorem}
\label{thm:lambdacoords} {\rm [Theorem 2.2.5 of \cite{[P3]}]}
For any quasi triangulation $E$ of $F$ based at a set ${\mathcal P}\neq\emptyset$ of punctures, the natural mapping $$\Lambda_E:\widetilde{T}_{\mathcal P}(F)\rightarrow \mathbb{R}_{>0}^{E}$$ given by
\[
\widetilde{\Gamma} \mapsto (e \mapsto \lambda(e;\widetilde{\Gamma}))
\]
is a real-analytic homeomorphism.
\end{theorem}

\begin{definition}[h-lengths]
Given a q.c.d.\ $E$ of $F$ and a point $\tilde\Gamma\in\widetilde T(F)$, the hyperbolic length of a sub-arc of a horocycle subtending an ideal polygon is its {\em h-length}.
\end{definition}

We shall typically let lower-case Roman letters denote lambda lengths and lower-case Greek letters denote h-lengths.  Formulas for h-lengths of horocyclic sub-arcs in terms of nearby lambda lengths in ideal triangles and once-punctured monogons are shown in Figure~\ref{fig:hlengthB}a.  Note that the hyperbolic length of the horocycle itself is by definition the sum of h-lengths of its  
horocyclic sub-arcs.

\begin{figure}[htbp]
	\centering
		\includegraphics[width=0.75\textwidth]{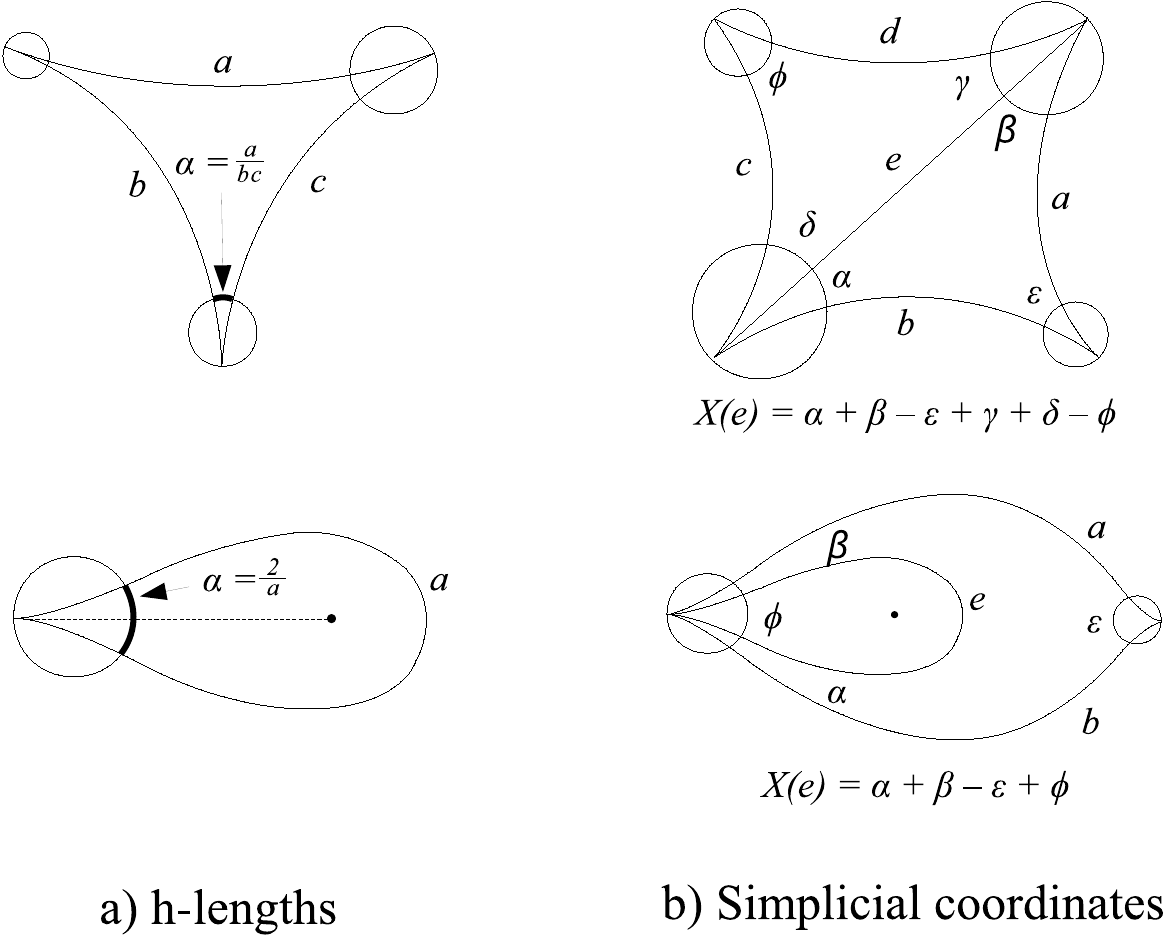}
	\caption{{\small  h-lengths and simplicial coordinates.}}
	\label{fig:hlengthB}
\end{figure}

\begin{definition}[Simplicial coordinates]
Suppose that $e$ is the diagonal of a decorated ideal quadrilateral with frontier edges $a,b,c,d$.  
The {\em simplicial coordinate} $X(e)$ of $e$ given by:
\[
X(e) = \alpha + \beta - \epsilon + \gamma + \delta - \phi=\frac{a^2 + b^2 - e^2}{abe} + \frac{c^2 +d^2 - e^2}{cde}
\]
again in the notation of Figure~\ref{fig:hlengthB}b top with
\[
X(e) =  \alpha + \beta - \epsilon  +  \phi = \frac{a^2 + b^2 - e^2}{abe} + \frac{2}{e}
\]
in case $e$ bounds a once-punctured monogon as in Figure~\ref{fig:hlengthB}b bottom. 
\end{definition}

The geometric significance
of simplicial coordinates will be discussed in subsection \ref{subsec:convex}.

It follows from Theorem \ref{thm:lambdacoords} that projectivized lambda lengths give coordinates on the
projectivized decorated Teichm\"uller spaces.  Simplicial coordinates and h-lengths
likewise descend to projective classes by homogeneity.
As a further point of notation, we may let $\lambda$, $\alpha$ and $X$ denote 
tuples of lambda lengths, h-lengths and simplicial coordinates 
with respective
projectivized tuples $\bar\lambda$, $\bar\alpha$, and $\bar X$.

\subsection{Dual punctured fatgraphs}

\begin{definition}[Punctured fatgraph and vertices]
A {\em fatgraph} is a graph together with a cyclic ordering of the half-edges incident on each vertex.  A {\em punctured fatgraph} $G$ is a fatgraph with certain of its vertices colored with a $\star$ icon, which we call {\em punctured vertices}. 
\end{definition}

We will consider only punctured fatgraphs in which univalent and bivalent vertices are necessarily punctured.  A punctured fatgraph $G$ then determines a punctured surface $F(G)$ as follows.  To each $k$-valent vertex of $G$ is associated an oriented ideal $k$-gon, which is once-punctured if and only if the corresponding vertex is punctured, and which may therefore be a monogon or bigon.  Sides of these polygons are identified in pairs, one such pair corresponding to each edge of $G$, so as to preserve orientations.  The resulting surface $F(G)$ is a punctured surface, where the punctures come in two varieties, either punctured vertices in $G$, or ideal boundary components arising from the gluing together of ideal polygons to form the surface. 

\begin{definition}[Boundary component of a punctured fatgraph]
By a {\em boundary component} of a punctured fatgraph $G$, we will mean an ideal boundary component arising from the construction of $F(G)$. 
\end{definition}

For each q.c.d.\ $E$ of $F$, there is a {\em dual punctured fatgraph} $G(E)$, whose vertices and edges, respectively, are in one-to-one correspondence with components of  $F-E$ and components of $E$.  Specifically, an edge connects two vertices of $G(E)$ if its corresponding dual ideal arc lies in the common frontier of the corresponding complementary regions, where the cyclic ordering is induced from the orientation of $F$; see Figure \ref{fig:StandardQuadB}.  This figure also illustrates that vertices associated to punctured complimentary regions are colored by $\star$.
In particular, vertices of $G(E)$ of valence one or two are necessarily punctured.
Any q.c.d.\  $E'$ may be completed to a quasi triangulation $E$ based at the same punctures, and $G(E')$ arises from $G(E)$ in this case by collapsing a forest of edges of $G(E)$, i.e., a collection of edges so that each component is contractible and contains at most a single punctured vertex.  Conversely as above, a punctured fatgraph $G$ determines a corresponding
q.c.d.\ $E(G)$ in the surface $F(G)$ it determines.  Given  $\widetilde{\Gamma}$, it is convenient to associate lambda lengths or simplicial coordinates  to elements of $E$ or their dual edges in $G(E)$, and likewise for h-lengths associated to sectors between edges.  Figure \ref{fig:StandardQuadB} indicates our standard notation for h-lengths nearby an edge of $G(E)$.  We shall sometimes abuse notation and let the Roman letters denote at once an ideal arc in $E$, its dual edge in $G(E)$, or its lambda length for some specified
 $\tilde\Gamma$.

\begin{figure}[htbp]
	\centering
		\includegraphics[width=0.650\textwidth]{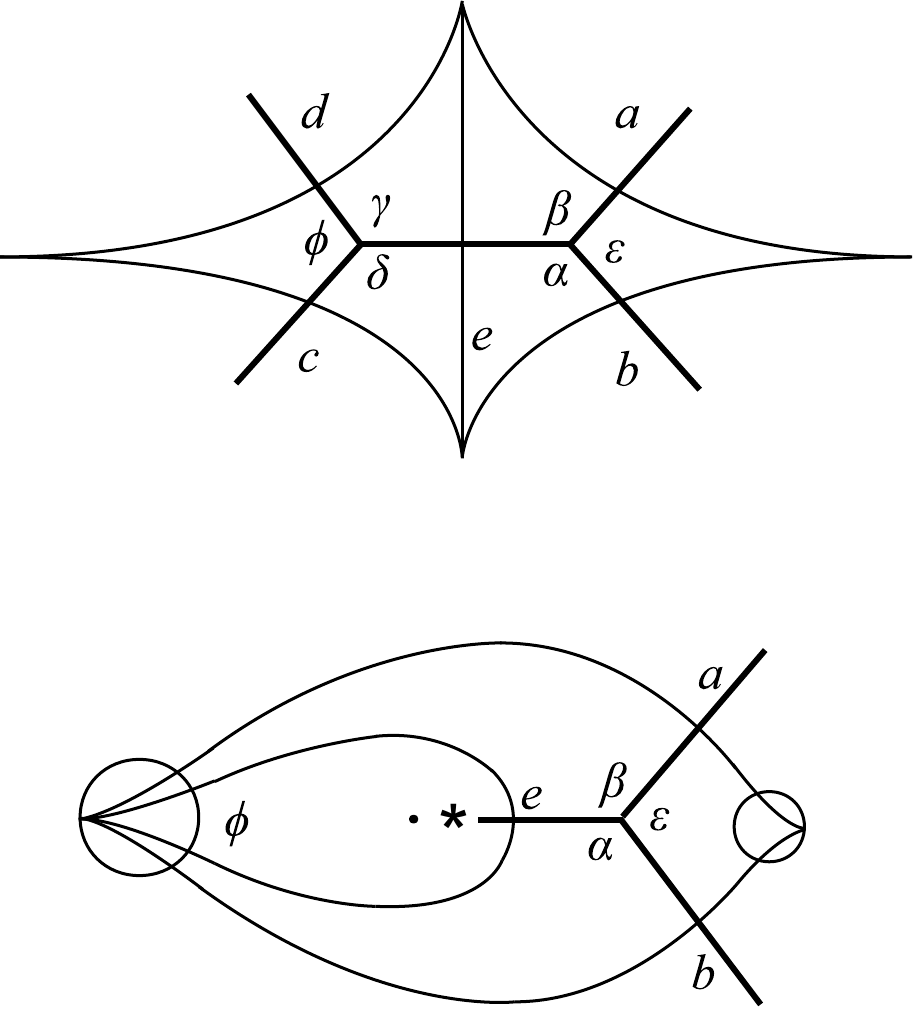}
	\caption{{\small Standard notation near an edge $e$.}}
	\label{fig:StandardQuadB}
\end{figure}

\begin{definition}[Quasi efficient, simple cycles and horocycles]
We say that a cycle  $\gamma$ in a possibly punctured fatgraph is {\em quasi efficient} if it 
never consecutively traverses an edge followed by its reverse unless the intervening vertex is punctured.
If a quasi efficient cycle $\gamma$ in a fatgraph is a chain of distinct edges having only bivalent vertices, none of which is punctured, we shall say that $\gamma$ is a {\em simple cycle}.  A simple cycle which is boundary-parallel to a horocycle on $F$ we will refer to as a {\em horocycle}.
\end{definition}

\begin{lemma} [Telescoping Lemma]
\label{lem:telescope}
Suppose that $\gamma$ is a quasi efficient cycle in a possibly punctured fatgraph consecutively traversing the edges $e_1,\cdots,e_n$ and let $\alpha_1,\cdots,\alpha_n$ be the included h-lengths between consecutive edges along $\gamma$.  Then
\[
\sum_{i=1}^n X(e_i) = \sum_{i=1}^n \alpha_i.
\]
\end{lemma}

\begin{proof}
This follows immediately from the formula for simplicial coordinates via local cancellation.
\end{proof}

\subsection{Convex Hull Construction and Cell Decomposition}
\label{subsec:convex}

For conceptual completeness, we finally briefly recall the convex hull construction
\cite{EP} in 
Minkowski 3-space (i.e.,
$\mathbb R^3$ with bilinear pairing $(x,y,z)\cdot(x',y',z')= xx' + yy' -zz'$) and the cell decomposition of
decorated Teichm\"uller spaces \cite{[P]}.  Affine duality $u\mapsto h(u)=\{ w\in \mathbb H: w\cdot u=-2^{-{1\over 2}}\}$,
where $\mathbb H=\{ w\in \mathbb R^3:w\cdot w=-1 \textrm{ and } z > 0\}$,
identifies $L^+=\{ u\in\mathbb R^3:u\cdot u=0 \textrm{ and } z > 0\}$ with the collection of all horocycles in
the model $\mathbb H$ of the hyperbolic plane.  

In fact, the lambda length between two horocycles $h(u)$ and $h(v)$ is computed
to be simply $\sqrt{-u\cdot v}$, so the identification with $L^+$ is geometrically natural.
Furthermore, the simplicial coordinate can be computed to be $2^{3\over 2}abcd$ times the volume of the tetrahedron
spanned by the vertices of a lift of a decorated quadrilateral to $L^+$ in the notation of Figure \ref{fig:StandardQuadB},
where a vertex in $L^+$ corresponding to an undecorated puncture is taken to infinity in $L^+$.  

Any punctured hyperbolic surface $F$ is manifest as $\mathbb{H}/\Gamma$, where  $\Gamma<SO(2,1)$ is a group of Minkowski isometries 
(respecting orientation and preserving $\mathbb H$).  Moreover, using the above identification of $L^+$ with horocycles in $\mathbb{H}$, a set of partial decorations (or horocycles) are a $\Gamma$-invariant discrete set $\mathcal B\subset L^+$ corresponding
to that particular partial decoration.  The Euclidean closed convex hull of $\mathcal B$ in $\mathbb R^3$ is shown to have only elliptic and parabolic support planes as conic sections, and the edges in the frontier of this convex body project in the natural way to a family $E(\tilde\Gamma)$ of arcs connecting punctures in the underlying surface $F=\mathbb H/\Gamma$.

\begin{theorem}
\label{thm:icd}{\rm [Theorem 5.5.9 of \cite{[P3]}]}
For any non-empty set $\mathcal P$ of punctures of $F$ and
any $\tilde{\Gamma} \in \widetilde{T}_{\mathcal P}(F)$, the collection
$E(\tilde{\Gamma})$ is a q.c.d.\ of $F$ based at ${\mathcal P}$.
\end{theorem} 

We emphasize that a puncture of $F$ is {\em not} in $\mathcal P$, and hence undecorated, if and only if it is represented by a punctured vertex in the dual punctured fatgraph $G(E(\widetilde{\Gamma}))$ arising from the convex hull construction.

Since the convex hull construction is invariant under homothety in Minkowski space, it is independent of global scaling of lengths of horocycles and therefore descends to the projectivized spaces. For any q.c.d.\ $E$ of $F$, we are led to define
$$\cC(E) = \{\widetilde{\Gamma}\in P \widetilde T_*(F) : E(\widetilde{\Gamma})=E\}~~\subset~~  \bar{\cC}(E) = \{\widetilde{\Gamma} \in P\widetilde T_*(F): E(\widetilde{\Gamma})\subseteq E\},
$$
where equality and inclusion of q.c.d.'s are understood only up to proper isotopy fixing the punctures. 

\begin{theorem}
\label{thm:celldecompdecTeich}{\rm [Theorem 5.5.9 of \cite{[P3]}]}
For any surface $F$ of negative Euler characteristic with $s >0$ punctures, the collection
\[
\{\cC(E) : E \textrm{ {\em is a q.c.d.\ of }} F\}
\]
of subsets is an $MC(F)$-invariant ideal simplicial decomposition for $P\widetilde{T}_*(F)$.  Indeed, 
for any quasi triangulation $E$, projectivized simplicial coordinates on arcs as barycentric coordinates give $\cC(E)$ the natural structure of an open simplex in $P\widetilde{T}_*(F)$, and this simplicial structure extends to $\bar{\cC}(E)$ by adding certain faces, namely, those faces corresponding to sub-arc families of $E$ with simplicial coordinates whose support is a q.c.d.\ of $F$.
\end{theorem}

\begin{remark}
\label{rem:triangle}
In fact, a point in $\cC(E)$ is determined by the lambda lengths on the edges of $E$, and these lambda lengths  $\lambda (e_i)$ on the frontier edges $e_i$, for $i=1,\ldots ,k$, of a complementary $k$-gon must satisfy the generalized triangle inequalities
$\lambda(e_i)\leq \sum_{j\neq i} \lambda(e_j)$ with strict inequality if the $k$-gon is unpunctured.  One can prove this directly
algebraically using non-negativity of simplicial coordinates on edges of $E$ or geometrically since a plane in Minkowski space
containing three points in $L^+$ is elliptic or parabolic, respectively, if and only if the lambda lengths between them satisfy the strict triangle inequalities or some triangle equality.
\end{remark}

Via the identification of punctured fatgraphs with q.c.d.'s, we shall sometimes equivalently
regard the cells $\cC(G)=\cC(G(E))$ or $\bar\cC(G)=\bar\cC(G(E))$ as indexed by isotopy classes of fatgraphs
embedded in $F$.  In this context, the face relation is generated by collapsing an edge with distinct endpoints at most one
of which is punctured, where if one endpoint is indeed punctured, then the resulting vertex is again punctured.

%************************************************************
\section{Space of filtered screens}
\label{sec:Screen}

We begin by recalling the principal definitions and results from \cite{[MP],[P3]} which are the starting point for
the considerations of this paper and then proceed to the key new definitions and their investigation.

\subsection{Screens}
\label{subsec:screens}

Consider a q.c.d.\ $E$ of a surface $F$ and its dual possibly punctured fatgraph $G=G(E)$.  A subset $A\subset E$ determines a corresponding smallest subgraph $G(A)$ containing the edges $A$.  Note that $G(A)$ may be disconnected and may have univalent or bivalent vertices which may be non-punctured.  

\begin{definition}[Quasi recurrent, recurrent]
We say that $A$ or $G(A)$ is {\it quasi recurrent} provided
each univalent vertex of $G(A)$ is punctured.  If $G(A)$ is quasi recurrent without punctured vertices, we simply say that $G(A)$ is {\em recurrent}\footnote{1-particle irreducible graphs from quantum field theory are those without univalent vertices that furthermore have no separating edges.  These play a role in renormalization that is analogous to the role played here by recurrent fatgraphs.}.
\end{definition}

\begin{lemma}
\label{lem:rec}
A subset $A$ of a q.c.d.\ is quasi recurrent if and only if for each $e\in A$ there is a quasi efficient cycle
in $G(A)$ which traverses $e$.
\end{lemma}

\begin{proof}
There can be no quasi efficient cycle traversing an edge incident on a univalent vertex which is not punctured.  Conversely if there are no univalent vertices which are not punctured,
we may take a regular neighborhood of $G(A)$ in $F$ so as to form a (possibly disconnected) subsurface of $F$ containing $G(A)$.  The resulting boundary of this subsurface, which is formed by connecting parallel copies of all of the edges contained in $A$, is a collection of simple closed curves.  Each component curve of the result represents a quasi efficient cycle in $G(A)$,
and for each edge in $A$, there is either one or two such curves traversing it by construction.
\end{proof}

\begin{corollary} Any subset $A$ of a q.c.d.\  contains a (possibly empty) maximal quasi recurrent subset given by the union of all quasi efficient cycles in $G(A)$.
\end{corollary}

\begin{definition}[Screen]
A {\em screen} $\cA$ on a q.c.d.\ $E$ or its dual fatgraph $G(E)$ is a collection of subsets of $E$ with the following properties:
\begin{itemize}
	\item $E \in \mathcal{A}$;
	\item each $A \in \mathcal{A}$ is quasi recurrent;
	\item if $A,B \in \mathcal{A}$ with $A \cap B \neq \emptyset$, then either $A \subseteq B$ or $B \subseteq A$;
	\item for each $A \in \mathcal{A}$, we have $\bigcup \{B \in \mathcal{A}\ : \ B \subsetneq A\}\subsetneq A$.
\end{itemize}
\end{definition}

For each $A\in \cA$, there  is a 
maximal chain of proper inclusions $A=A^n \subset A^{n-1} \subset \cdots \subset A^0 = E$, where each $A^i \in \cA$,
and we define the {\it depth} of $A$ to be $n$.  Each edge $e\in E$ has its {\it depth} defined to be $\max\{\textrm{depth of }A\in\cA : e \in A\}$.
In fact, each element $A \in \mathcal{A}$ other than $A=E$ has an immediate predecessor $A'$.
We may furthermore consider the subsurface $F(G(A)) \subset F(G(A'))$, and define the {\em relative boundary} $\partial_\mathcal{A} A$ to be those boundary components of $F(G(A))$ which are neither puncture-parallel nor homotopic to a boundary component of $F(G(A'))$ (see Figure \ref{fig:New2B}); moreover, if $G(A)$ includes a simple cycle component which is not a component in $G(A')$, we simply include that simple cycle in $\partial_\mathcal{A} A$ (rather than two parallel copies arising from the boundary of the annular subsurface).  Finally, the {\em boundary} of the screen itself is $\partial\mathcal{A} = \bigcup_{A \in \mathcal{A}-\{E\}} \partial_\mathcal{A} A$. 

\begin{figure}[htbp]
	\centering
		\includegraphics[width=0.95\textwidth]{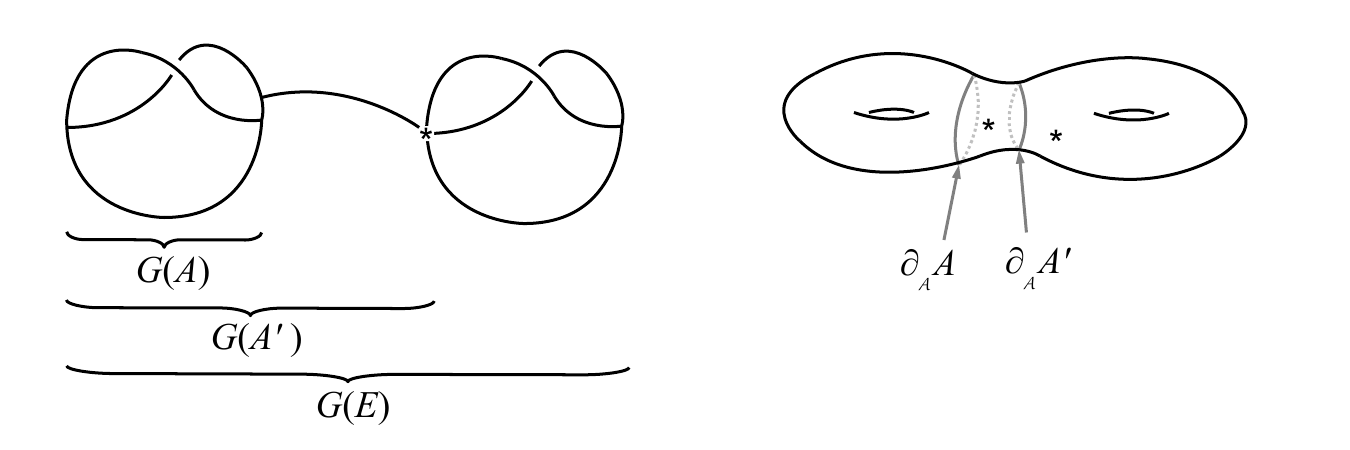}
	\caption{{\small A screen $\{A, A', E\}$ along with its boundary.}}
	\label{fig:New2B}
\end{figure} 

\begin{definition}[Stable path]
Suppose that $f_t:E\to{\mathbb R}_{>0}$, i.e., $f_t\in {\mathbb R}_{>0}^E$, for $t>0$, is a continuous one-parameter family of functions on the edges of some q.c.d.\ $E$.  Letting $P$ denote projectivization, there is an induced $\bar f_t\in P({\mathbb R}_{>0}^E)$, and
compactness of  $P{\mathbb R}_{\geq 0}^E\supset P({\mathbb R}_{>0}^E)$ guarantees the existence of accumulation points of ${\rm lim}_{t\to\infty} \bar f_t$ in $P{\mathbb R}_{\geq 0}^E$.  We say that $f_t$ is {\it stable} provided there is a unique such limit point and for any $e_1,e_2 \in E$ that $\lim_{t \to \infty} \frac{f_t(e_1)}{f_t(e_2)}$ exists, where here we allow the possibility that $\lim_{t \to \infty} \frac{f_t(e_1)}{f_t(e_2)}=\infty$.
\end{definition}

\begin{theorem}[McShane-Penner, \cite{[MP],[P3]}]
\label{thm:McShanePenner}
The cell $\cC(G)$ in $P\widetilde T_{\cP}(F)$ corresponding to the fatgraph $G$ contains a stable path of lambda lengths whose projection to $T(F)$ is asymptotic to a stable curve with pinch curves $\sigma$ if and only if $\sigma$ is homotopic to the collection of edge-paths $\partial \mathcal{A}$ for some screen $\mathcal{A}$ on $G$.
\end{theorem}

The various members of a screen ${\mathcal A}$ correspond to subsets of edges whose lambda lengths diverge at least as fast as others, where $A\subset B\in {\mathcal A}$ if and only if the lambda lengths on edges in $A$ diverge faster than those in $B-A$; on the other hand, a screen does not keep track of the relative rates of divergence of disjoint elements.  In fact, it is the triangle inequalities discussed in Remark \ref{rem:triangle} that obviously force the recurrence condition.  The difficult part of the proof of the previous theorem is estimating the lengths of the curves in $\partial \mathcal{A}$ in terms of lambda lengths.

%------------------------------------------------------------
\subsection{Filtered screens}
\label{subsec:filtscreen}

We next introduce a refinement of screens which plays the central role in this paper, where in contrast to screens, we keep track of rates of vanishing of simplicial coordinates along quasi recurrent subsets of a q.c.d., and we record relative rates of vanishing of disjoint subsets.
We shall consider quasi recurrent subsets $A$ of a q.c.d.\  $E$ and their dual graphs $G(A)$, which as before  may have bivalent vertices which are not punctured.  Lambda lengths and simplicial coordinates of an edge $e\in A$  or $G(A)$ are defined to be those associated to $e$ in the original fatgraph $G(E)$.

\begin{definition}[Filtered screen]
\label{defn:filtscreen}
A {\em filtered screen of total level $n$} on a q.c.d.\ $E$ or its dual fatgraph $G(E)$ is an ordered $(n+1)$-tuple 
$\vec{\cE} = (L^0,L^1,\cdots,L^n)$
of pairwise disjoint and non-empty subsets of $E$ so that:
\begin{itemize}
	\item $\bigcup_{i=0}^n L^i =E$;
	\item for each $0 \leq k \leq n$, $L^{\geq k}=\bigcup_{i=k}^n L^i$ is quasi recurrent.
\end{itemize}
\noindent An arc $e\in L^k$ is said to have {\it level} $k$.
\end{definition}

\noindent Notice that a filtered screen $\vec{\cE}$ in particular determines its underlying q.c.d.\ $E=L^{\geq 0}$.  Moreover, a filtered screen of total level zero on $E$ is simply the q.c.d.\  $\vec{\cE}=(E)$ regarded as a one-tuple.

\begin{definition}[Boundary of a filtered screen]
In analogy to screens, for each $0 \leq k < n$, we have the subsurface $F(G( L^{\geq k+1})) \subset F(G(L^{\geq k}))$ and define the
{\em $(k+1)$-st relative boundary} $\partial_{k+1}\vec{\cE}$ to be those boundary components of $F(G(L^{\geq k+1}))$ which are neither puncture-parallel nor homotopic to a boundary component of $F(G(L^{\geq k}))$; if a component of $G(L^{\geq k+1})$ is a simple cycle which is not a component in $G(L^{\geq k})$, then include just that cycle in $\partial_{k+1}\vec{\cE}$.   The {\em boundary} of $\vec{\cE}$ itself is defined to be $\partial \vec{\cE} = \cup_{i=1}^n \partial_{i}\vec{\cE}$. 
\end{definition}

Given a filtered screen $\vec{\cE}$, there is an {associated screen}
\[
\cA(\vec{\cE})=\{L^{\geq k} : 0 \leq k \leq n\}.
\]
Although $\vec{\cE} = \vec{\cE}'$ if and only if $\cA(\vec{\cE}) = \cA(\vec{\cE}')$, 
not all screens arise in this way, only those which have a unique member at each depth.
Conversely, to each screen $\cA$ on $E$ is associated a filtered screen $\vec{\cE}(\cA)$
where the level of each $e\in E$ is its depth; the resulting 
$$L^{\geq k}=\{ e\in E:{\rm the~depth~of}~e~{\rm is~at~least}~k\}$$ are quasi recurrent because if $e \in L^{\geq k}$, then $e$ lies in a quasi recurrent subset $A\subseteq E$ of depth $i$ for some $k \leq i \leq n$, whence also $A \subset L^{\geq k}$.  

%Insofar as $\partial \cA = \partial \vec{\cE}(\cA)$ and $\partial \vec{\cE} = \partial\cA(\vec{\cE})$ by construction, we have the following reformulation of Theorem \ref{thm:McShanePenner}.

%\begin{corollary}
%\label{thm:FSversionMcShanePenner}
%The cell $\cC(G)$ in $P\widetilde T_{\cP}(F)$ corresponding to the fatgraph $G$ contains a stable path of lambda lengths whose projection to $T(F)$ is asymptotic to a stable curve with pinch curves $\sigma$ if and only if $\sigma$ is homotopic to the collection of edge-paths $\partial \vec{\cE}$ for some filtered screen $\vec{\cE}$ on $G$.
%\end{corollary}

This discussion is somewhat misleading since we shall employ filtered screens to record rates of vanishing of simplicial coordinates
rather than rates of divergence of lambda lengths as for screens, and these notions have yet to be related, cf.\ Proposition \ref{prop:stablepathyieldscreen}.

\begin{figure}[htbp]
	\centering
		\includegraphics[width=0.65\textwidth]{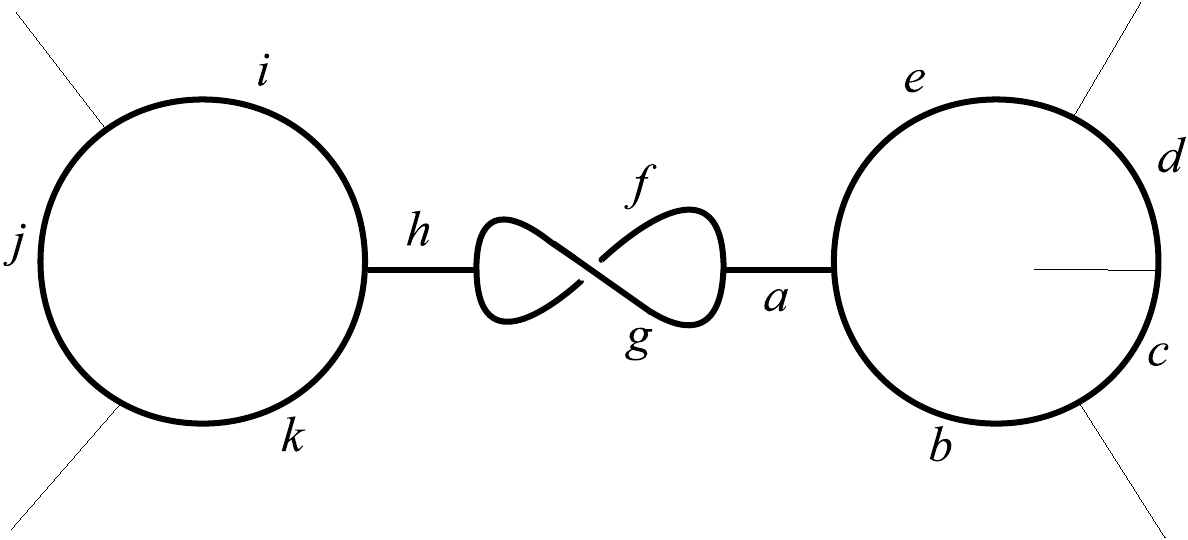}
	\caption{{\small A subfatgraph of $G(E)$.}}
	\label{fig:FiltScreensExampleB}
\end{figure}

\begin{example}
\label{ex:filtscreen}
Consider the subfatgraph of a fatgraph $G(E)$ depicted in Figure \ref{fig:FiltScreensExampleB},
and adopt the notation that ``$m-n$'' means the set of labeled edges lexicographically between $m$ and $n$ and including $m,n$.
The two screens
\[
\cA = \{E, a-k, a-g, i-k, f-g, b-e\}
\]
and
\[
\cA' = \{E, a-k, (a-g)\cup(i-k), (f-g)\cup(b-e)\}
\]
are distinct,
however, $\vec{\cE}(\cA) = \vec{\cE}(\cA') = \vec{\cE}$, where the filtered screen is
\[
\vec{\cE} = (E-(a-k), h, a\cup(i-k), (b-e)\cup(f-g)).
\]
\end{example}

%----------------------------------------------------------------------------------
\subsection{The space of filtered screens}
\label{subsec:spaceFS}

In this section, we construct a simplicial complex whose simplices are indexed by filtered screens on q.c.d.'s based at $\cP$ for a fixed surface $F$.
Let us adopt the notation that for filtered screens $\vec{\cE}$ and $\vec{\cE}'$, the respective sets of arcs of level
$k$ are denoted $L^k$ and $(L')^k$.

\begin{definition}
\label{defn:cellassocfiltscreen}
Given a filtered screen $\vec{\cE}$ of total level $n$, the {\em (open) cell associated to $\vec{\cE}$} is a product of (open) simplices defined as  
\[
\cC(\vec{\cE}) = P\mathbb{R}_{>0}^{L^0} \times P\mathbb{R}_{>0}^{L^1} \times \cdots \times P\mathbb{R}_{>0}^{L^n}
\]
 
\noindent We have included the word {\em open} in the above definition because (as will be made precise in the next definition) the boundary of this cell is not the ordinary boundary of a product of simplices, but rather a blowup of the ordinary boundary.

\end{definition}

\begin{definition}
\label{defn:spaceoffilteredscreens} For any surface $F$, define the
{\em space of filtered screens based at $\cP$} 
\[
\cF\cS_{\cP}(F)= \left[\bigsqcup_{{{\textrm{filtered screens }\vec{\cE}\textrm{ on some}}\atop{{\cP-{\textrm {based q.c.d.\ of}}~F}}}}
\cC(\vec{\cE})\right]/\sim
\]
where the face relation $\sim$ is generated by the following:

\begin{itemize}
\item $\cC(\vec{\cE'})$ is a codimension-one face of $\cC(\vec{\cE})$ if there exists an edge $e \in L^k \subset E$ whose dual in $G(E)$ has distinct endpoints, at least one of which is non-punctured and disjoint from $G(L^{\geq k+1})$, and such that $(L')^k = L^k - \{e\}$ and $(L')^j = L^j$ for $j \neq k$.   
In other words, we pass from $\vec{\cE}$ to $\vec{\cE'}$ by removing such an arc from $E$; alternatively, we collapse the dual edge in $G(E)$.
Removing such an arc corresponds to setting the associated simplicial coordinate to zero in the factor simplex corresponding to $L^k$.
	
\medskip
	
\item $\cC(\vec{\cE'})$ is a codimension-one face of $\cC(\vec{\cE})$ if $\vec{\cE}$ arises from  $\vec{\cE'}$ by combining two adjacent levels and shifting up by one all greater levels, i.e.,  we have
$$
\hskip 6ex \vec{\cE} = ((L')^0,(L')^1,\cdots,(L')^{k-1},(L')^k \cup (L')^{k+1},(L')^{k+2},\cdots,(L')^n),
$$
for some $0 \leq k < n$.  In terms of coordinates, moving from $\cC(\vec{\cE})$ to $\cC(\vec{\cE'})$ corresponds to choosing a quasi recurrent subset $(L')^{\geq k+1}$ such that $L^{\geq k+1} \subsetneq (L')^{\geq k+1} \subsetneq L^{\geq k}$, and then allowing the simplicial coordinates of $(L')^{\geq k+1} - L^{\geq k+1}$ in $P\mathbb{R}_{>0}^{L^k}$ to go to zero all at comparable rates.  In the limit, $(L')^k=L^{\geq k} - (L')^{\geq k+1}$ and $(L')^{k+1}=(L')^{\geq k+1} - L^{\geq k+1}$ are separately projectivized.
\end{itemize}
\end{definition} 

\noindent These face relations will be explicated in the sequel where we demonstrate that they precisely describe the asymptotics of stable paths constrained to lie in cells in $\cP$-decorated Teichm\"{u}ller space.

We shall denote by $\cF\cS(F)$ the space of filtered screens on i.c.d.'s of $F$ for which all punctures are decorated and refer to $\cF\cS(F)$ simply as the {\em space of filtered screens on $F$}.
Note that the cell decomposition for $\cF\cS(F)$ is $MC(F)$-invariant by construction.
The quotient
$\cF\cS(F)/MC(F)$ consists of finitely many cells, since there are finitely many fatgraphs on $F$ modulo the action of $MC(F)$, and then finitely many ways to partition edges on each of these fatgraphs.  The quotient $\cF\cS(F)/MC(F)$ therefore provides a new compactification of Riemann's moduli space ${\mathcal M}(F)$.

\begin{example}
Figure \ref{fig:PPDTB+}a depicts the cell decomposition of $P\widetilde T(F_0^3)$, where the small circles and dashed lines respectively represent zero- and one-simplices that are absent from $P\widetilde T(F_0^3)$.   
Figure \ref{fig:PPDTB+}b illustrates the space of filtered screens for $F_0^3$, where the added faces
are labeled by the corresponding filtered screens.

\begin{figure}[htbp]
	\centering
		\includegraphics[width=0.85\textwidth]{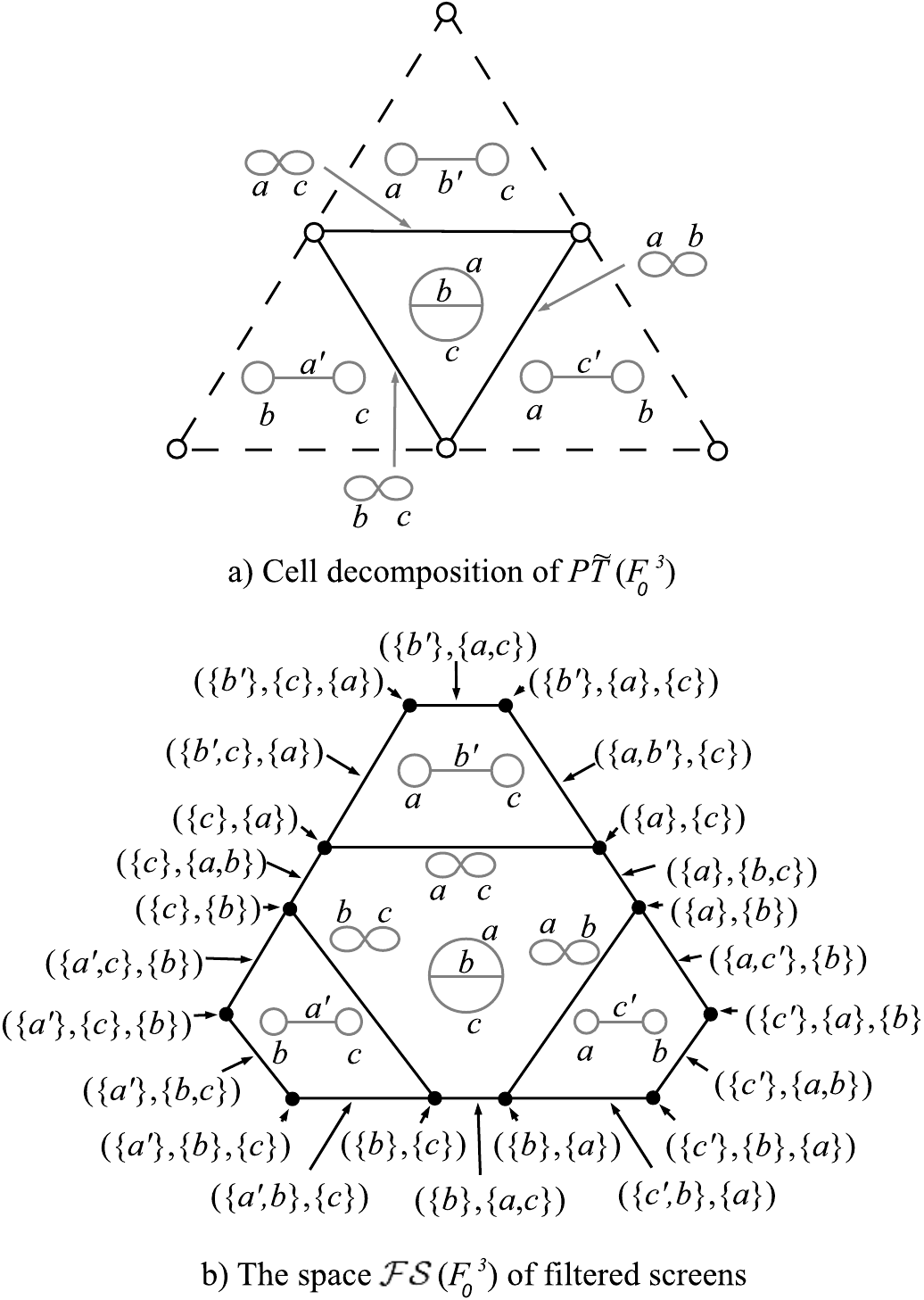}
	\caption{{\small Filtered screen blow up of cell decomposition for $F_0^3$.}}
	\label{fig:PPDTB+}
\end{figure}
\end{example}

%---------------------------------------------------------------------
\subsection{Stable paths and the Filtered IJ Lemma}
\label{subsec:ijlemma}

Fix a q.c.d.\ $E$ of a surface $F$ and consider its corresponding cell $\cC(E)\subset P\widetilde{T}_{\cP}(F)$.  
To each point of $\cC(E)$ is associated its tuple of projectivized simplicial coordinates and lambda lengths on the edges of $E$.  
By adding arcs if necessary to quasi triangulate regions complementary to $E$, we may and shall assume that $E$ is a quasi triangulation.
While the simplicial coordinates of the added arcs vanish, their lambda lengths do not.

We may thus regard  a continuous path in $\cC(E)$ as its associated continuous one-parameter family of projectivized simplicial coordinates  $\bar{X}_t(e)\geq 0$ or lambda lengths $\bar{\lambda}_t(e) >0$, defined for each edge $e \in E$ and all $t\geq 0$.
It will sometimes be convenient to de-projectivize these paths in various ways to produce corresponding one-parameter families
${X}_t(e)\geq 0$ and ${\lambda}_t(e) >0$, which we may also regard as paths in the deprojectivized $\cC(E)$ lying in $\widetilde T_{\cP}(F)$ itself.

Our primary immediate goal is to show that to each stable path of simplicial coordinates in $\cC(E)$, there is a canonical assignment of a point  in a cell $\cC(\vec{\cE})$ associated with a filtered screen based on a q.c.d.\ contained in $E$ which is  also based at $\cP$.  This will be accomplished in the next sub-section after first here recalling and slightly extending several results from \cite{[P3]}.

\begin{lemma}
\label{lemma:triangleinequalities}
Given any point in a deprojectivized $\bar\cC(E)$ for any quasi triangulation $E$, the associated simplicial coordinates on $E$ satisfy the {\rm no quasi vanishing cycle condition}: there is no quasi efficient cycle of edges in $G(E)$ each of whose simplicial coordinates vanishes.
\end{lemma}
\begin{proof}
This follows from the fact that the convex hull construction of a $\cP$-decorated surface produces a q.c.d.\ 
based at $\cP$ by Theorem \ref{thm:icd}.  Specifically, given a point in $\bar\cC(E)$, the convex hull construction shows that the simplicial coordinates for edges in $G(E)$ will be positive for a subset of $E$ which is a q.c.d.  This q.c.d.\ then divides the surface $F$ into polygons bounded by arcs with positive simplicial coordinates.  As a result, no quasi efficient cycle with vanishing simplicial coordinates can exist.
\end{proof}

\begin{lemma}
\label{lemma:upperboundonproduct}
Given any point in a deprojectivized $\cC(E)$ and any arc $e$ in the quasi triangulation $E$, the product of the lambda length of $e$ and the simplicial coordinate of $e$ is bounded above by four.
\end{lemma}

\begin{proof}
If $e$ is not incident on a punctured vertex, in the notation of Figure \ref{fig:StandardQuadB} we have
\[
eX(e) = \frac{a^2 + b^2 - e^2}{ab} + \frac{c^2 +d^2 - e^2}{cd}
\]
so each summand is twice a cosine by Remark \ref{rem:triangle} and the Euclidean law of cosines.  If $e$ is incident on a univalent punctured vertex, in the notation of Figure \ref{fig:StandardQuadB} we have
\[
eX(e) = \frac{a^2 + b^2 - e^2}{ab} + 2
\]
which is again bounded above by four.
\end{proof}

\begin{lemma}
\label{lemma:lambdalengthsboundedbelow}
Fix a point in a deprojectivized $\cC(E)$ and suppose that $a,b,e$ are the frontier edges of a triangle complementary to $E$
with respective opposite h-lengths $\alpha,\beta,\epsilon$.   If $1\leq a,b,e$ and $\alpha\leq K<1$, then we have
$b,e\geq \frac{1}{2\sqrt{K}}$.
\end{lemma}

\begin{proof}
We have $\alpha = \frac{a}{be} \leq K$ and $a \geq 1$ whence $K \geq \frac{a}{be} \geq \frac{1}{be}$, and so

\begin{equation}
\label{eqn:firstinlemma}
\frac{1}{b}\frac{1}{e} \leq K.
\end{equation}
Furthermore by Remark \ref{rem:triangle}, we have $|b-e| < a$, and so dividing by $be$ we find also
\begin{equation}
\label{eqn:secondinlemma}
|\frac{1}{e}-\frac{1}{b}| < \frac{a}{be} \leq K.
\end{equation}
It follows from inequality (\ref{eqn:firstinlemma}) that either $\frac{1}{b} \leq \sqrt{K}$ or $\frac{1}{e} \leq \sqrt{K}$.  Without loss of generality concentrating on the former case and using inequality (\ref{eqn:secondinlemma}), we conclude
\[
\frac{1}{e} \leq \frac{1}{b} + K \leq \sqrt{K} + K \leq 2\sqrt{K}
\]
as was claimed.
\end{proof}

\begin{definition}
\label{defn:IJ}
Given a deprojectivized path $\lambda_t$ in a deprojectivized $\cC(E)$, and a subset $H \subset E$, define
\[
I(H) = \{ e \in H: \lim_{t\rightarrow\infty}\lambda_t(e) = \infty\}
\]
and
\[ 
J(H) = \{ e \in H: \lim_{t\rightarrow\infty}X_t(e) = 0\}
\]
\end{definition}

The following result is of central importance in the next section and specializes to the so-called IJ Lemma of \cite{[MP],[P3]}
for $H=E$:

\begin{lemma}[Filtered IJ Lemma]
\label{lemma:FilteredIJLemma}
Suppose that $E$ is a quasi triangulation with dual punctured fatgraph $G(E)$.  Let $H \subset E$, with $G(H)$ a quasi recurrent subfatgraph component that is not a simple cycle; that is, $G(H)$ may have bivalent vertices but has at least one trivalent or punctured vertex.
Suppose that a deprojectivized path in $\cC(E)$ has $\lambda_t(f)>1$ for all $f \in G(H)$ and that each edge $e$ of $G(E-H)$ incident on a bivalent vertex of $G(H)$  has $\lambda_t(e)\to 0$.  Then $I(H) \subset J(H)$, and $G(I(H))$ is the maximal quasi recurrent subfatgraph of $G(J(H))$.
\end{lemma}

\begin{proof}
If $\lambda_t(f) \rightarrow \infty$, then $X_t(f) \rightarrow 0$ by Lemma \ref{lemma:upperboundonproduct}, so we indeed have $I(H) \subset J(H)$.

To see that the maximal quasi recurrent subfatgraph of $G(J(H))$ is contained in $G(I(H))$, recall that
 a subfatgraph is quasi recurrent if and only if for each of its edges, there is a quasi efficient cycle traversing the edge by Lemma
\ref{lem:rec}.  If $\gamma$ is a quasi efficient cycle in $G(J(H))$, then $X_t(e) \rightarrow 0$ for each edge $e$ traversed by $\gamma$
by definition.  According to the Telescoping Lemma \ref{lem:telescope}, we have $\sum X_t(e) = 2\sum \alpha_t(e)$, so the sum of the h-lengths included in $\gamma$ goes to zero.  If an edge $e$ in $\gamma$ is incident on a univalent punctured vertex, then since $\lambda(e) = 2/\alpha(e)$, where $\alpha$ is the included h-length, we see that $e \in G(I(H))$.  Now, consider a vertex in $\gamma$ with edges $a,b,e$ incident upon it, where $\gamma$ traverses $b,e$ and includes the h-length $\alpha$ opposite $a$.
In case all three of $a,b,e \in H$, then by Lemma \ref{lemma:lambdalengthsboundedbelow}, we have $\lambda_t(b),\lambda_t(e) \rightarrow \infty$ since $\alpha \rightarrow 0$ as required.  
It remains therefore to consider the case that $a \notin H$, so that since $G(H)$ is not a simple cycle,
$b$ and $e$ must lie in a linear chain of edges connecting vertices bivalent in $H$ that begins and ends at vertices either trivalent or punctured univalent in $H$.
Since all incident  edges not in $H$ along this chain have vanishing lambda lengths by assumption, all of the lambda lengths on the edges of this linear chain must have ratios asymptotic to unity by the triangle inequalities.
Since the first and last edges in this chain must have lambda lengths which diverge by earlier remarks, in fact all of these edges must have lambda lengths which diverge.  Thus, each edge traversed by $\gamma$ lies in $G(I(H))$ as required.

For the reverse inclusion that $G(I(H))$ is contained in the maximal recurrent subfatgraph of $G(J(H))$, suppose that $e$ is an edge with $\lambda_t(e) \rightarrow \infty$.  As before,  $X_t(e) \rightarrow 0$, so $e\in G(J(H))$.  Thus, we need only show that
there is a quasi efficient cycle in $G(J(F))$ traversing $e$.  To this end, if the endpoints of $e$ coincide, then there is such a cycle of length one.  Otherwise by the triangle inequalities, there is an edge $f$ adjacent to $e$ at either of its non-puncture endpoints such that $\lambda_t(f) \rightarrow \infty$, so $X_t(f) \rightarrow 0$ by Lemma \ref{lemma:upperboundonproduct}.  Continuing in this manner, we find a quasi efficient cycle lying in $G(I(F))$ passing through $e$, which evidently also lies in $G(J(F))$ as desired.
\end{proof}

%--------------------------------------------------------------------
\subsection{Stable paths and filtered screens}
\label{subsec:stablepathfiltscreen}

We shall next canonically associate a point of $\cF\cS_{\cP}(F)$ to a stable path in $\cC(E)\subset P\widetilde T_{\cP}(F)$.

\begin{definition}
\label{def:comparison}
Fix a stable path in $\cC(E)$ with associated deprojectivized simplicial coordinates $X_t(e)$,
for $t\geq 0$ and $e\in E$.  The simplicial coordinates of two edges $e,f\in E$ are said to be {\em comparable} if 
$
0 < \lim_{t\rightarrow\infty}\frac{X_t(e)}{X_t(f)} < \infty
$
and be {\em asymptotically equal} 
provided
$
\lim_{t\rightarrow\infty}\frac{X_t(e)}{X_t(f)} = 1.
$
The corresponding terms of asymptotic comparison among lambda lengths are defined analogously.
\end{definition}

\begin{definition}[Comparability filtration]
Suppose that $X_t$, for $t\geq 0$, is a deprojectivization, which is bounded above, of a stable path in $\bar\cC(E)$ for some q.c.d.\ $E$. Recursively define a nested collection $E=E_0\supset E_1\supset\cdots\supset E_n$ of sets as follows.  If $X_t(e_k) \neq 0$ for some $e_k \in E_k$, choose some $e_k \in E_k$ so that $\lim_{t \to \infty} \frac{X_t(e_k)}{X_t(e)} > 0$ for all $e \in E_k$; define $E_{k+1}\subset E_k$
to be those edges that are not comparable to $e_k$.
Thus, any two elements of
$E_k-E_{k+1}$ are comparable, for $k=0,\ldots ,n$, where we set $E_{n+1}=\emptyset$ for convenience.   The collection $E_0\supset E_1\supset\cdots \supset E_n$ is evidently invariant under overall scaling of simplicial coordinates and is called the
{\em comparability filtration} on $E$ induced by the stable path $\bar X_t$.
\end{definition} 

\begin{definition}[Recurrence hierarchy]
Given a stable path $\bar X_t$ in $\bar\cC(E)$ for some q.c.d.\ $E$, along with its comparability filtration $E=E_0\supset E_1\supset\cdots\supset E_n$, observe that each $E_j$ has a maximal quasi recurrent subset $R_j$; we then define $L_j=R_j \cap (E_j - E_{j+1})$.  This gives an initial $n$-tuple $(L_0,L_1,\cdots,L_n)$ which we further modify as follows:  if $L_j = \emptyset$ we remove it from the tuple and relabel $L_{j+i}$ to be $L_{j+i-1}$ for $i > 0$.  After removing all such empty sets, the result is a {\it recurrence hierarchy} $(L_0,L_1,\cdots, L_k)$ of non-empty sets for some $k \leq n$, associated with the given stable path of simplicial cooordinates.
\end{definition}

%Observe that two different comparability filtrations may be assumed to have the same length by adjoining empty sets to the end of the chain of inclusions if necessary.

We come to our first important result whose proof will serve two purposes:  first, it will show that each stable path $\bar X_t$ in $\cC(E)\subset P\widetilde T_{\cP}(F)$ limits to a unique {\em limiting filtered screen} $\vec{\cE}(\bar X_t)$ associated to a $\cP$-based q.c.d.\ contained in $E$, as well as to a unique {\em limiting point} $p(\bar X_t) \in \cC(\vec{\cE}(\bar X_t))$; second, it will show that $\bar X_t$ produces a screen of comparability classes of lambda lengths whose boundary matches the boundary of $\vec{\cE}(\bar X_t)$, and thus proves that the short curves associated with $\bar X_t$ are precisely those at the boundary of the filtered screen.

\begin{proposition}
\label{prop:stablepathyieldscreen}
Let $E$ be a q.c.d.\ based at $\cP$, and let $\bar{X}_t$ be a stable path in the cell $\cC(E)\subset P\widetilde T_{\cP}(F)$ corresponding to some q.c.d. $E$ of the surface $F$.  Then there is a canonically determined limit point $p(\bar{X}_t)\in \cF\cS_{\cP}(F)$ in the cell $\cC(\vec{\cE}(\bar{X}_t))$ of a corresponding filtered screen $\vec{\cE}(\bar{X}_t)$ on a $\cP$-based q.c.d.\ contained in $E$ with the following properties:

\begin{itemize}
\item Two such paths $\bar X_t$ and $\bar X_t'$ in a common $\cC(E)$  have the same associated filtered screens
$\vec{\cE}(\bar{X}_t)=\vec{\cE}(\bar{X}_t')$ if and only if their respective recurrence hierarchies $(L_0,\cdots,L_k)$ and $(L'_0,\cdots,L'_k)$ are equal.

\smallskip 

\item
Moreover, the stable paths have the same associated points $p(\bar{X}_t)=p(\bar{X}_t')$ if and only if
they further have the same limiting ratios of simplicial coordinates on 
$L_j=L'_j$, for each $j=0,\ldots ,k$.
\end{itemize}
\end{proposition}

\begin{proof}
Choose a quasi triangulation $E''$ containing $E$, so that $\mathcal{C}(\mathcal{E}) \subseteq \overline{\mathcal{C}}(\mathcal{E}'') \subset P\widetilde{T}_{\mathcal{P}}(F)$.  We shall algorithmically construct $p(\bar X_t)\in \cC(\vec{\cE}(\bar X_t))$ via a recursion, which will use repeated applications of the Filtered IJ Lemma and be shown to be independent of the choice of $E''$.  For the purposes of this proof,  given a set $H$ of edges in a q.c.d.\
such that the deprojectivized $X_t(f)$ vanishes for each $f \in H$, let us define
\[
H^\infty = \{f \in H | \lim_{t\rightarrow\infty} \frac{X_t(f)}{\sup_{e\in H} X_t(e)} = 0\}.
\]
In other words, $H^\infty$ consists of those edges whose simplicial coordinates go to zero faster than the slowest in $H$.

To begin the recursion, define the set $D_{0} = E''$ and the sets $B_{0} = Z_{0} = C_{0} = \emptyset$.  Let us choose a deprojectivization $\lambda _t$ of the path in lambda lengths corresponding to $\bar X_t$ so that $\lambda_t(e) >  1$, for each $e\in D_0$.  Such a deprojectivization exists, since even if $\bar\lambda_t(e)$ goes to zero for some $ e \in D_0$, we may rescale all lambda lengths by the reciprocal of a function which goes to zero just as quickly as any of $\bar\lambda_t$.

The Filtered IJ Lemma thus applies to $H=D_0=E''$ in $E''$ (since $E''$ has only univalent punctured vertices with all other vertices trivalent) and yields $I(D_{0})\subseteq J(D_{0})$.  We separate the subfatgraph $G(I(D_{0}))$ into its components and label the associated collections of edges  $I_1,\cdots,I_m$.  The general recursion producing new sets $B_{i+1},D_{i+1} Z_{i+1},C_{i+1}$ is given by
\[
B_{i+1} = D_{i} - J(D_{i})
\]

\[
D_{i+1} = \bigcup_{\textrm{non-simple cycles}\ I_k} I_k
\]

\[
Z_{i+1} = \bigcup_{\textrm{simple~cycles}~I_k} (I_k - I_k^\infty)
\]

\[
C_{i+1} =   [J(D_{i})-I(D_{i})]~~\cup~~ \bigcup_{\textrm{simple~cycles}\ I_k}I_k^\infty,
\]
Let us pause to examine this output for the basis step $i=0$.  First of all, $B_1$ consists of all arcs in $D_0$ which have simplicial coordinates bounded below and lambda lengths bounded above for all $t$.  As a result, the simplicial coordinates of the arcs in $B_1$ are all comparable 
and likewise their lambda lengths.  We further observe that since any edges of $E''-E$ have $\bar X_t = 0$, they are deepest in the comparability filtration and hence not in $B_1$.

$D_1$ consists of the quasi recurrent $I_k$ that are not simple cycles which have divergent lambda lengths and hence vanishing simplicial coordinates; this then becomes input for the next stage in the process.  We emphasize that the edges of $G(D_1)$ have divergent lambda length compared to any adjacent edges not in $G(D_1)$.

$Z_1$ consists of the duals of edges in each simple cycle with divergent, asymptotically equal, lambda lengths but slowly vanishing simplicial coordinate; thus, the intersection of $Z_1$ with a component $I_k$ consists of edges with comparable lambda lengths and comparable simplicial coordinates.  Upon collapsing
duals of $\cup _{\rm simple~cycles} I_k^\infty\subseteq C_1$, each component $I_k$ which intersects with $Z_1$ again becomes a simple cycle composed of the edges $(I_k -I_k^\infty)$.  At the end of our procedure, we shall determine to which level of the filtered screen these cycles belong; again, we observe that edges of $E''-E$ are not in $Z_1$.  

Finally, $C_1$ consists of two subsets.  The first, $J(F_{0})-I(F_{0})$, is made up of  those edges whose simplicial coordinates go to zero, but whose lambda lengths remain bounded; the duals of these edges will be appropriately collapsed at the end of the algorithm.  The second subset arises from simple cycles with comparably divergent lambda lengths; these edges, which are again finally collapsed,   make up
linear chains whose simplicial coordinates go to zero faster than other simplicial coordinates in their cycle.  Thus, these edges have comparable lambda lengths but vanishing simplicial coordinates and will also be appropriately collapsed at the end of the algorithm.  To make sure that these collapses are legitimate, we have:

\smallskip

\noindent{\bf Claim}~
$G(C_1)$ is a forest in $G(E)$ so that each component tree meets $G(D_1)$ in at most a single endpoint
of an edge.

\smallskip

\noindent {\bf Proof.}~Indeed, each component of $G(J(D_0)-I(D_0))$ is a tree since $I(D_0)$ is the maximal quasi recurrent subset of $J(D_0)$ by the Filtered IJ Lemma.
Each such tree can have at most one point in common with the $G(I_k)$ since if there were two such points of intersection, then there would be a quasi efficient cycle in $G(J(D_0))$ that was not in $G(I(D_0))$, again contradicting the maximality of $I(D_0)$.  This completes the proof of the claim.  

\smallskip

Proceeding now to the general iteration, we are recursively provided with a quasi recurrent $D_i$ so that $G(D_i)$ has no simple cycle components
and whose edges have divergent lambda lengths compared to any adjacent edges of $G(E)$ not in $G(D_i)$.  Again we may rescale all lambda lengths so that the edges of $D_i$ have scaled lambda lengths bigger than one, but so that edges incident on but not in $G(D_i)$ will have rescaled lambda lengths tending to zero.  The hypotheses of the Filtered IJ Lemma are therefore satisfied, and we apply that lemma to $D_i$ to obtain $I(D_i)$ and $J(D_i)$.  As before, we may separate $I(D_i)$ into its components $I_1,\cdots,I_m$ and define a partition of $D_i$ according to the recursion.

The analysis of the resulting sets $B_{i+1}$, $Z_{i+1}$, $D_{i+1}$, and $C_{i+1}$ is entirely analogous to the previous discussion for the basis step; again we observe that edges of $E''-E$ are not in $B_{i+1}$ or $Z_{i+1}$, and we furthermore have:

\smallskip

\noindent{\bf Claim}~$G(\cup_{j=1}^{i+1} C_j)$ is a forest in $G(E)$ so that each component tree meets $G(D_{i+1})$ in at most a single endpoint of an edge.
\smallskip

\noindent {\bf Proof.}~To see this, we may assume by induction the previous claim that $G(\cup_{j=1}^{i} C_j)$  is a forest, and the only way components of $G(\cup_{j=1}^{i} C_j)$ can meet $G(D_i)$ is at a single endpoint of an edge. 
It follows that indeed  $G(\cup_{j=1}^{i+1} C_j)$ is a forest, since $G(C_{i+1})$ is a forest in $G(D_i)$ by the Filtered IJ Lemma.
Moreover, the component trees of $G(C_{i+1})$ can only possibly meet  the new $G(I_k)$ at a single vertex, for otherwise, $I(D_i)$ could not be the maximal quasi recurrent subset of $J(D_i)$.  As a result, if a component tree of $G(\cup_{j=1}^{i+1} C_j)$ were to meet the new $G(I_k)$ at two of its endpoints, only one of these could occur at a point where $G(C_{i+1})$ intersects the $G(I_k)$; the other must occur where $G(\cup_{j=1}^{i} C_j)$ intersects the $G(I_k) \subset G(D_i)$.  This latter point plus the point where $G(\cup_{j=1}^{i} C_j)$ intersects $G(C_{i+1}) \subset G(D_i)$ would then represent two distinct points which $G(\cup_{j=1}^{i} C_j)$ has in common with $G(D_i)$, which 
is impossible by the inductive hypothesis.  This completes the proof of the claim.

\smallskip

Let us assume that the recursion terminates with $D_{n+1}=\emptyset$ having constructed
the various subsets $C_i, Z_i, B_i$ of $E$, for $i=1,\ldots ,n$. We must next
build the resulting filtered screen $\vec{\cE}=\vec{\cE}(\bar X_t)$ and identify the associated point $p=p(\bar X_t) \in \cC(\vec{\cE})$.

To this end, the collection $E'=E''-\cup_{j=1}^{n} C_j \subseteq E$ is a q.c.d.\ according to the previous claim on which we shall define the filtered screen.
The various sets $Z_i,B_i$, for $i=1,\ldots ,n$ are pairwise disjoint, and the arcs in any connected component of $Z_i$ or $B_i$ have comparable simplicial coordinates by construction.
We may thus group together the arcs in $\bigcup_{i=1}^n Z_i\cup \bigcup_{i=1}^n B_i$
into common comparability classes taking the union of these sets in a given class to define
the edges at a fixed level $L_j$ in our filtered screen $\vec{\cE}$ on $E'$.  Quasi recurrence is assured from the construction, since each  $\bigcup_{j=i}^k L_j$ is in fact equal to one of the quasi recurrent sets $D_i$ minus a contractible set of edges, plus possibly additional quasi recurrent simple cycles;
finally, the point $p\in \cC(\vec{\cE}(\bar{X}_t))$ is determined by limiting ratios of projective classes of simplicial coordinates in each level.

It follows by construction that $p$ and $\vec{\cE}$ have the characterizing properties in the statement of the proposition
since the recursion itself depends only on comparability classes of arcs in each maximal quasi recurrent set. In other words, at each step of the recursion it is the non-empty $L_j = R_j \cap (E_j - E_{j+1})$ which form levels in the resulting filtered screen.  In particular, since $E''-E$ is a forest for any quasi triangulation $E''$ containing $E$,  for any subset $A \subset E$, there is a one-to-one correspondence between quasi efficient cycles in $A \cup (E''-E)$ and $A \cup ({E_1}''-E)$ for any two quasi triangulations $E''$ and ${E_1}''$ containing $E$.  Moreover, since the maximal recurrent subset of a set is simply the union of all quasi efficient cycles in the dual fatgraphs of that set, taking the maximal recurrent subset for each of $A \cup (E''-E)$ and $A \cup ({E_1}''-E)$ and then collapsing edges in $E''-E$ and ${E_1}''-E$ yields the maximal recurrent subset of $A$ in either case.  Thus, our construction does not depend on the choice of quasi triangulation.
\end{proof}

\begin{remark}
Observe that the above proposition provides a simple algorithm for producing a filtered screen from a stable path of simplicial coordinates, namely collapse all edges that are not in the maximal quasi recurrent subset for their respective subset $E_j$ in the comparability filtration, and then set as levels of the filtered screen the remaining edges in each $E_j-E_{j+1}$.  However, the detail of the preceding proof allows us to establish our first main theorem which justifies using filtered screens to identify short curves in the limit of such stable paths.
\end{remark}

\begin{theorem}
\label{thm:FSversionMcShanePenner}
The cell $\cC(G)$ in $P\widetilde T_{\cP}(F)$ corresponding to the fatgraph $G$ contains a stable path of simplicial coordinates whose projection to $T(F)$ is asymptotic to a stable curve with pinch curves $\sigma$ if and only if $\sigma$ is homotopic to the collection of edge-paths $\partial \vec{\cE}$ for some filtered screen $\vec{\cE}$ on $G$.
\end{theorem}

\begin{proof}

In the notation of the preceding proof, there is a
screen 
$$\cA(\bar{\lambda}_t) = \{E,I(D_0),I(D_1),\dots,I(D_n)\}$$
which is naturally associated to a stable path $\bar{\lambda}_t$ corresponding
to $\bar X_t$.  We have  $\partial \vec{\cE}(\bar{X}_t) = \partial \cA(\bar{\lambda}_t)$, and by Lemma 4.38 in \cite{[MP]}, these are precisely the curves whose hyperbolic lengths tends to zero.
\end{proof}

We now show in particular that each point of $\cF\cS_{\cP}(F)$ arises as the limiting point of some stable path in a cell of $P\widetilde T_{\cP}(F)$:

\begin{proposition}
\label{prop:everypointinscreenhasstablepath}
Given q.c.d.'s $E'\subseteq E$ and a point $p \in \cC(\vec{\cE'})$ for some filtered screen $\vec{\cE'}$ on $E'$, there is a stable path $\bar{X}_t$ in $\cC(E)$ such that $p=p(\bar{X}_t)$ and $\vec{\cE}(\bar{X}_t) = \vec{\cE'}$.   
\end{proposition}

\begin{proof}
Suppose that $\vec{\cE}'=(L^0,\ldots ,L^n)$ is of total level $n$ and let $x_k(e)$ denote deprojectivized
 simplicial coordinates on $e\in L^k$ in the $k$th factor simplex $P\mathbb{R}_{>0}^{L^k}$ of $\cC(\vec{\cE}')$,
 for $k=0,\ldots, n$.
 Complete $E$ to an ideal triangulation $E''$ and define for $t\geq 1$
$$X_t(e)=\begin{cases}
0,& {\rm if}~e\in E''-E;\\
x_k(e)~t^{-k},& {\rm if}~ e_i\in L^k, ~{\rm for}~ 0\leq k\leq n;\\
t^{-(n+1)},& {\rm if} ~e\in E-E'.\\
\end{cases}$$
The projectivization $\bar X_t$ is a stable path in ${\cC}(E)$ that by construction has the specified limiting point and filtered screen.
\end{proof}

To conclude the discussion, we explain the face relations in Definition \ref{defn:spaceoffilteredscreens} in terms of comparability filtrations of stable paths.  Let $\vec{\cE}=(L^0,\cdots,L^n)$ be a fixed filtered screen based on a q.c.d.\ $L^{\geq 0}$, and complete $L^{\geq 0}$ to a quasi triangulation $E$ by adding a set $B$ of ideal arcs.  We first observe as a consequence of Proposition \ref{prop:stablepathyieldscreen} that stable paths $\bar X_t$ in $\cC(L^{\geq0}) \subset \bar\cC(E)$ having comparability filtration $$E=(L^{\geq 0}\cup B) \supset (L^{\geq 1}\cup B) \supset \cdots \supset (L^n\cup B) \supset B$$ are precisely those with $\vec{\cE}(\bar X_t) = \vec{\cE}$.  In the following two corollaries of Proposition \ref{prop:stablepathyieldscreen}, we consider limiting filtered screens of stable paths in $\cC(E)$ which slightly increase the comparability class for a proper subset of edges in a level $L^k$.

\begin{corollary}
Let $\vec{\cE}=(L^0,\cdots,L^n)$ be a filtered screen based on a q.c.d.\ $L^{\geq 0}$, and complete $L^{\geq 0}$ to a quasi triangulation $E$ by adding a set $B$ of ideal arcs.  Let $e \in L^k$ be an arc whose dual edge in $G(E)$ has distinct endpoints, at most one of which is punctured.  Consider any stable path $\bar X'_t$ in $\cC(E)$ having comparability filtration 

\begin{eqnarray*}
(L^{\geq 0} \cup B)\supset \cdots &\supset& (L^{\geq k}\cup B) \supset (L^{\geq k+1} \cup \{e\}\cup B) \supset (L^{\geq k+1}\cup B)\\ &\supset& \cdots \supset (L^n\cup B) \supset B.
\end{eqnarray*}

\noindent Then we have the following: 

\begin{itemize}
	\item[1.]   if the dual edge of $e$ in $G(L^{\geq 0})$ has a non-puncture endpoint disjoint from $G(L^{\geq k+1})$, then 
	$$\vec{\cE}(\bar X'_t) = (L^0,\cdots,L^{k-1},L^k-\{e\},L^{k+1},\cdots,L^n)$$ 
	is based on the q.c.d.\ $L^{\geq 0}-\{e\}$;
	\item[2.] if the dual edge of $e$ has both endpoints in $G(L^{\geq k+1})$ and $L^k - \{e\}$ is nonempty, then 
	$$\vec{\cE}(\bar X'_t) = (L^0,\cdots,L^{k-1},L^k-\{e\},\{e\},L^{k+1},\cdots,L^n).$$
\end{itemize}  
\end{corollary}

\begin{proof}
First observe that if the edge dual to $e$ in $G(E)$ is contained in a simple cycle component $I$ of $L^{\geq k}$, then $e$ necessarily satisfies the hypothesis of 1, and we also must have $L^k -\{e\}$ nonempty.  Thus, when we rescale the deprojectivization $X'_t$ of $\bar X'_t$ in the $k$-th step of the recursion in Proposition \ref{prop:stablepathyieldscreen} so that edges in $L^{\geq k}-(\{e\} \cup L^{\geq k+1})$ have bounded simplicial coordinates, we shall have $I^\infty = \{e\}$.  It follows that $\{e\} \in C_{k+1}$ is removed from the underlying q.c.d., and the limiting filtered screen  is then evidently $\vec{\cE}(\bar X'_t) = (L^0,\cdots,L^{k-1},L^k-\{e\},L^{k+1},\cdots,L^n)$ as required.
As a result, we may assume for convenience that $L^{\geq k}$ has no simple cycle components.  

For 1, we rescale at the $k$-th step of the recursion as above and observe that $J(L^{\geq k}\cup B) = \{e\} \cup L^{\geq k+1}\cup B$ and $L^{\geq k+1} \subset I(L^{\geq k}\cup B)$.  However, $e \notin I(L^{\geq k}\cup B)$, for if there is a quasi efficient cycle in $\{e\} \cup L^{\geq k+1}\cup B$ passing through $e$, then  all non-puncture vertices of $e$ in $G(L^{\geq 0})$  would be in $L^{\geq k+1}$.  Thus, $\{e\} \in C_{k+1}$ is removed from the underlying q.c.d.\ and limiting filtered screen, which is again $\vec{\cE}(\bar X'_t) = (L^0,\cdots,L^{k-1},L^k-\{e\},L^{k+1},\cdots,L^n)$.  For this last statement, we remark that since $L^{\geq k}$ is itself quasi recurrent, there must in this case be another edge $e'$ in $L^k$ whose dual in $G(E)$ shares with the dual of $e$ the endpoint disjoint from $G(L^{\geq k+1})$; thus, indeed $L^k -\{e\}$ is nonempty.

For 2, we again rescale at the $k$-th step and observe that in this case $\{e\} \cup L^{\geq k+1} \subset I(L^{\geq k}\cup B)$ since now $\{e\} \cup L^{\geq k+1}$ is recurrent.  Moreover, since by hypothesis $L^k-\{e\}$ is nonempty, we have $B_{k+1} = L^k-\{e\}$.  Furthermore, in the $(k+1)$-st step of the recursion, the deprojectivization $X'_t$ of $\bar X'_t$ is rescaled so that $e$ has bounded simplicial coordinate whence $B_{k+2} = \{e\}$, and the limiting filtered screen is $\vec{\cE}(\bar X'_t) = (L^0,\cdots,L^{k-1},L^k-\{e\},\{e\},L^{k+1},\cdots,L^n)$.
\end{proof}

\begin{corollary}
Let $\vec{\cE}=(L^0,\cdots,L^n)$ be a filtered screen based on a q.c.d.\ $L^{\geq 0}$, and complete $L^{\geq 0}$ to a quasi triangulation $E$ by adding a set $B$ of ideal arcs.  Let $A \subsetneq L^k$ be a proper subset with $A \cup L^{\geq k+1}$ quasi recurrent.  If a stable path $\bar X'_t$ in $\cC(E)$ has comparability filtration 

\begin{eqnarray*}
(L^{\geq 0} \cup B)\supset \cdots &\supset& (L^{\geq k}\cup B) \supset (L^{\geq k+1} \cup A\cup B) \supset (L^{\geq k+1}\cup B)\\ &\supset& \cdots \supset (L^n\cup B) \supset B.
\end{eqnarray*}

\noindent then $$\vec{\cE}(\bar X'_t) = (L^0,\cdots,L^{k-1},L^k-A,A,L^{k+1},\cdots,L^n).$$  
\end{corollary}

\begin{proof}
The proof is identical to the reasoning for case 2 in the preceding  lemma, replacing $\{e\}$ with $A$.
\end{proof}

%**************************************************************
\section{Decorated augmented Teichm\"uller space}
\label{sec:DecaugTeich}

With Proposition \ref{prop:stablepathyieldscreen} and Theorem \ref{thm:FSversionMcShanePenner} in hand, in this section we describe the quotient of $\cF\cS(F)$ which will yield the decorated augmented Teichm\"uller space $\widehat T(F)$.  Closely following the treatments \cite{[Brock],[W]}, we first recall the augmented Teichm\"uller space $\bar T(F)$ of a surface $F=F_g^s$, introduced in \cite{[Bers]}, which provides a bordification of $T(F)$.

%-----------------------------------------------------------------
\subsection{Augmented Teichm\"{u}ller space}
\label{subsec:augteich}

A maximal curve family $\cP=\{\alpha_i\}_1^N$, where the interior of each complementary region is a planar surface of Euler characteristic -1,  is a {\it pants decomposition} of $F$ and contains $N=3g-3+s$ curves.
Given a point $\Gamma\in T(F)$, there are associated hyperbolic lengths $\ell_i>0$ of the $\Gamma$-geodesic representatives of
the $\alpha_i$; these are completed to {\it Fenchel-Nielsen coordinates} on $T(F)$ by adjoining for each $\alpha _i$ one real {\it twisting parameter}
$\theta_i$ which records the relative displacement of the hyperbolic structures across the geodesic curves, for $i=1,\ldots , N$.

If $\sigma\subseteq \cP$, then there  is a corresponding stratum $T_\sigma$  added to $T(F)$ in  $\bar T(F)$ as follows:  the range of the Fenchel-Nielsen
coordinates with respect to $\cP$ is extended to allow $\ell_i\geq 0$, for  $\alpha_i\in\sigma$, where the twisting parameter $\theta_i$ is undefined
for $\ell_i=0$.  Geometrically, one pinches each curve in $\sigma$ to produce a nodal surface $F_\sigma$, and the added stratum $T_\sigma$ is a copy of the products of the Teichm\"uller spaces of all the irreducible components of $F_\sigma$.  In particular, $T_\emptyset\approx T(F)$ is the {\it null stratum}.

\begin{definition}
The {\it augmented Teichm\"uller space} of $F$ is the space
$$\bar T(F)=T_\emptyset\cup \bigcup T_\sigma,$$
where the union is over all non-empty curve families $\sigma$ in $F$, topologized so that
a neighborhood of a point $\bar F$ in $T_\sigma$, for $\sigma$ a subset of a pants decomposition
$\cP$,  consists of those points whose hyperbolic lengths
are close to those of $\bar F$ for each curve in $\cP$ and whose twisting parameters
are close to those of $\bar F$ for each curve in $\cP-\sigma$.
\end{definition} 

A mapping class  $\varphi\in MC(F)$ acts on $\bar T(F)$ by push-forward of metric as usual on the null stratum and induces
an action on other strata $T_\sigma\to T_{\varphi (\sigma)}$ via permutation of curve families and push-forward of metric
on irreducible components.  In particular, the pure mapping class group of each irreducible component acts on its
corresponding factor Teichm\"uller space in each stratum.

\begin{definition}[Stratum graph]
Associated with each stratum $T_\sigma$ is its {\em stratum graph}  $\cG(\sigma)$ which has a vertex indicated by a $\bullet$ icon 
for each irreducible component of $F_\sigma$ and an edge for each component of $\sigma$ itself connecting the vertices
corresponding to the irreducible components (which may coincide) on its two sides;  for each puncture of the original surface
$F$, we add yet another edge to $\cG(\sigma)$ connecting a univalent vertex indicated by the special icon $\circ$ with the vertex corresponding to the 
irreducible component of $F_\sigma$ containing it.  See Figure \ref{fig:StratumGraphB}.
\end{definition}

\begin{figure}[htbp]
	\centering
		\includegraphics[width=0.60\textwidth]{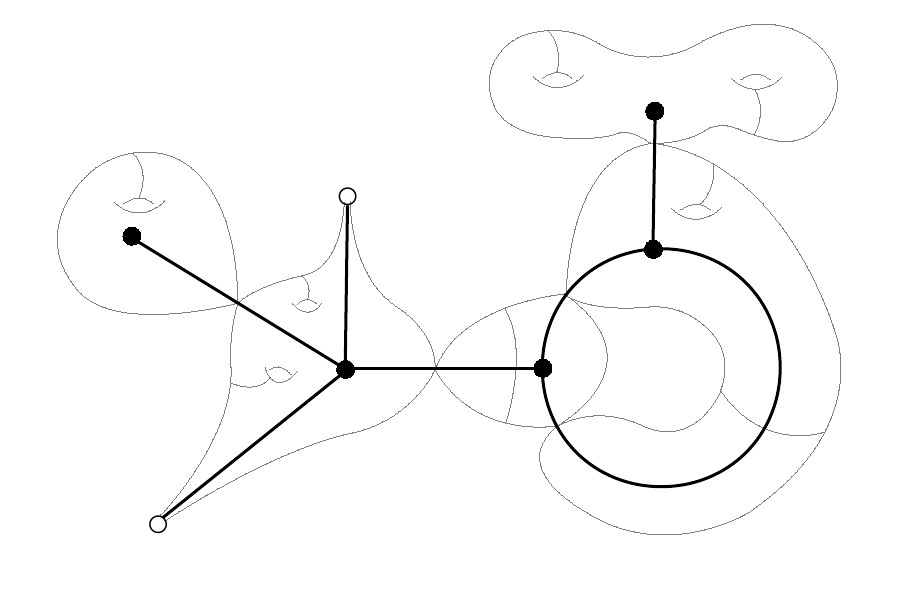}
	\caption{{\small The stratum graph $\cG(\sigma)$ associated with the nodal surface $F_\sigma$.}}
	\label{fig:StratumGraphB}
\end{figure}

\noindent We shall typically identify the edges of a stratum graph with $N \cup P$, where $N$ is the set of nodes and $P$ is the set of punctures of the nodal surface.  By construction, a stratum graph is always connected and contains at least one $\circ$-vertex.  It follows that there is a path
in the stratum graph from any element of $N\cup P$ to a $\circ$-vertex.

Given a stratum $T_\sigma \subset \bar T(F)$ and an $\alpha_i \in \sigma$ with $\ell_i = 0$, we can increase $\ell_i$, and in so doing, produce the stratum $T_{\sigma -\{\alpha_i\}}$ via the resolution of the node of $F_\sigma$ corresponding to $\alpha_i$, in which we replace the double point with $\alpha_i$.  On the level of stratum graphs, this resolution corresponds to collapsing an edge $x$ of $\cG(\sigma)$ to obtain $\cG(\sigma-\{\alpha_i\})$, with the proviso that if an edge has only one endpoint, we simply remove it from the graph.

\subsection{Punctured fatgraphs with partial pairing}
\label{subsec:partial}

Before performing the quotient we need one more object.

\begin{definition}[Punctured fatgraph with partial pairing] A {\it partial pairing} on a possibly disconnected punctured fatgraph is a collection of unordered pairs $\{\alpha,\beta\}$
where at least one of $\alpha$ or $\beta$ is a punctured vertex, and the other is 
either a distinct punctured vertex or a boundary component of a component fatgraph so that each punctured vertex or boundary component occurs at most once; see Figure \ref{fig:PairingsB}.  We shall refer to $\alpha$ and $\beta$ as {\em paired punctures} and denote the resulting {\em punctured fatgraph with partial pairing} by $\bar G$.\end{definition}

\begin{figure}[htbp]
	\centering
		\includegraphics[width=0.65\textwidth]{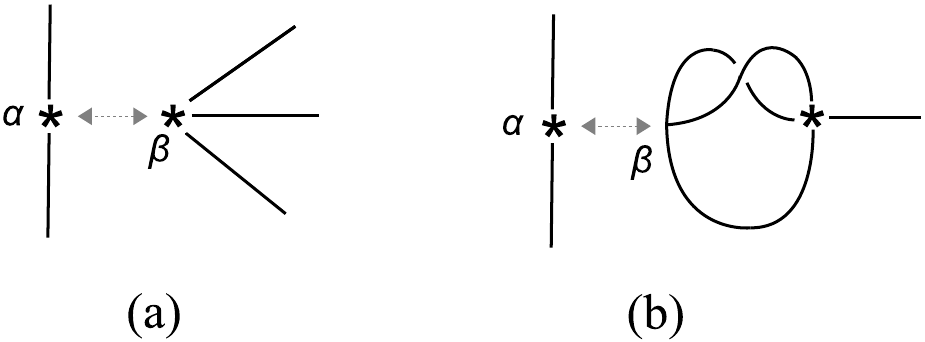}
	\caption{{\small The two types of paired punctures; in (a) two punctured vertices are paired, and in (b) a punctured vertex $\alpha$ is paired with a boundary component $\beta$.}}
	\label{fig:PairingsB}
\end{figure}

Fix a possibly disconnected surface $F$ with negative Euler characteristic $\chi(F)<0$ and $s>0$ punctures.  A {\em curve family} $\sigma$ in $F$ is (the isotopy class of) a collection of curves disjointly embedded in $F$, no component of which is puncture-parallel or null-homotopic, and no two of which are parallel.  Each such $\sigma$ determines a nodal surface $F_\sigma$ obtained by pinching each curve in $\sigma$ to a double point or node and then separating the resulting nodes into pairs of punctures to produce a surface whose components $F_1,\cdots,F_n$ are called the {\em irreducible components of} $F_\sigma$.   

\begin{definition}\label{def:support}
We say that a punctured fatgraph $\bar G$ with partial pairing  
and component fatgraphs $G_1,\ldots G_n$ is {\em supported by the connected surface $F$} if the following hold:
\begin{itemize}
	\item[i.] there are exactly $s$ unpaired punctures;
	\item[ii.] $\chi(F(G_i))<0$, for each $i = 1,\cdots,n$;		
	\item[iii.] $\chi(F) = \sum_{i=1}^n \chi(F(G_i))$;
	\item[iv.] any two component fatgraphs are connected by a sequence of partial pairings, i.e., for any two distinct punctured fatgraph components $G$ and $G'$, there is a sequence of components $G= G^{1},G^{2},\cdots,G^{m}=G'$ of $\bar G$ such that $G^{\ell}$ has a puncture or boundary component paired with a puncture or boundary component of $G^{{\ell+1}}$, for each $\ell$.  
\end{itemize}
\end{definition}

\noindent Figure \ref{fig:New1B} shows an example of a punctured fatgraph $\bar G$ with partial pairing supported by $F_2^4$, that is, a genus-2 surface with 4 punctures.

\begin{figure}[htbp]
	\centering
		\includegraphics[width=0.65\textwidth]{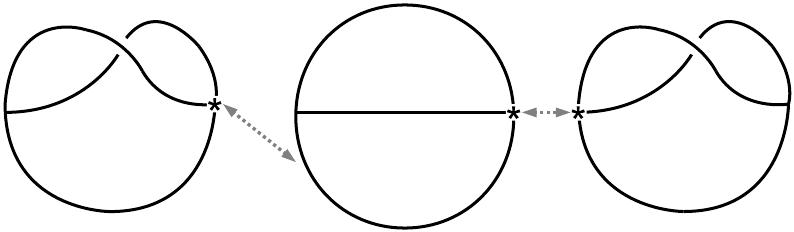}
	\caption{{\small An example of a punctured fatgraph $\bar G$ with partial pairing; the partial pairing on the left pairs a punctured vertex with a boundary component, and the partial pairing on the right pairs two punctured vertices.}}
	\label{fig:New1B}
\end{figure}

$\bar G$ thus determines the topological type of a nodal surface $F(\bar G)$, with the $F(G_i)$ as irreducible components and paired punctures corresponding to double points.  We may thus consider isotopy classes in $F$ of punctured fatgraphs $\bar G$ with partial pairing supported by $F$, by which we mean that $\bar G$ carries a specification of both a curve family $\sigma$ in $F$ such that $F(\bar G) = F_\sigma$, as well as an isotopy class of fatgraph $G_i$ in each irreducible component $F(G_i)=F_i$ of $F_\sigma$.
Moreover, we further consider {\em weighted} isotopy classes of punctured fatgraphs with partial pairing, which we still denote by $\bar G$, on which is specified a collection of positive weights, one weight for each edge, so that the sum of weights on edges within each component is one.  Two such labeled $\bar G$ and $\bar G'$ are therefore equivalent if and only if their respective curve families $\sigma$ and $\sigma'$  and component fatgraphs are isotopic, and the weights on their edges are identical.  Equivalently, one imagines projectivizing weights separately on each component fatgraph.

%------------------------------------------------------------------------------------
\subsection{The quotient space $\widehat T(F)$}
\label{subsec:equivalenceonfiltscreens}

There is an isotopy class in $F$ of weighted punctured fatgraph with partial pairing canonically associated to each 
point in $\cF\cS_{\cP}(F)$ as follows:

\begin{construction}  
\label{constructit}
Given a filtered screen
$\vec{\cE}=(L^0,\ldots ,L^n)$, consider its corresponding tower 
$$L^{\geq 0}\supset L^{\geq 1}\supset\cdots\supset L^{\geq n}$$
of quasi recurrent sets.  There is a corresponding tower of inclusions
$$G(L^{\geq 0})\supset G(L^{\geq 1})\supset\cdots\supset G(L^{\geq n}),$$
where each connected component of $G(L^{\geq k})$, for each $k=0,1\ldots ,n$, is either a simple cycle or it is not; if it is not, for the purposes of the construction we will call it a fatgraph, and we will further subdivide simple cycles into those that are puncture parallel (which we call horocycles) and those that are not (which we will just refer to as simple cycles).   These three possibilities for a component of 
$G(L^{\geq k+1})$ in $G(L^{\geq k})$ are then illustrated
in bold stroke in Figure \ref{fig:PartialPairingsB},
together with a modification of $G(L^{\geq k})$
to be performed in each case in order to produce a partial pairing
indicated by double arrows on the resulting punctured
fatgraph.  These modifications are performed in 
order of decreasing level beginning with $G(L^{\geq n}) \subset G(L^{\geq n-1})$ so as to produce an isotopy class of punctured fatgraph $\bar G(\vec{\cE})$
with partial pairing associated to $\vec{\cE}$.

\begin{figure}[htbp]
	\centering
		\includegraphics[width=0.85\textwidth]{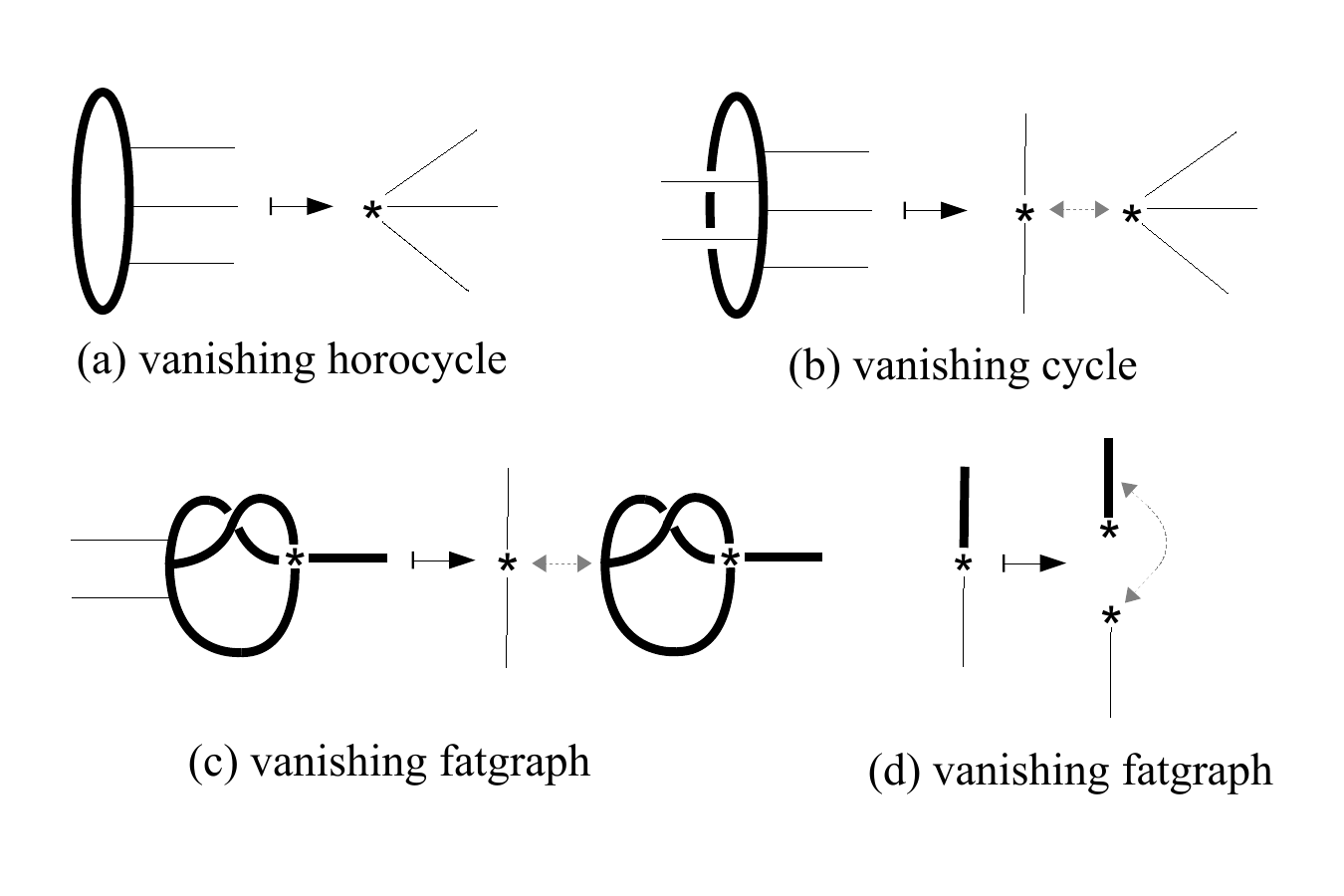}
	\caption{{\small The four operations performed on a filtered screen on a fatgraph to obtain the corresponding punctured fatgraph with partial pairing.}}
	\label{fig:PartialPairingsB}
\end{figure}

Furthermore, given a point $p\in\cC(\vec{\cE})$,  there are associated projectivized simplicial coordinates induced on
the edges of each component punctured fatgraph of $\bar G(\vec{\cE})$ as follows:
First, separately deprojectivize each factor simplex in 
$\cC(\vec{\cE})$ in order to have  coordinates defined on each edge.
\begin{itemize}
	\item For edges arising from operations (a), (b) or (d) in Figure \ref{fig:PartialPairingsB}, take the same coordinate in 
$\bar G(\vec{\cE})$.
	\item For edges arising from operation (c) in Figure \ref{fig:PartialPairingsB}, take as coordinate the sum of the  coordinates of the constituent edges that have combined to form the new edge.
\end{itemize} 
\noindent Finally, projectivize these coordinates on each punctured fatgraph component of $\bar G(\vec{\cE})$
to obtain projective simplicial coordinates on each component of $\bar G(\vec{\cE})$.
\end{construction}

\begin{figure}[htbp]
	\centering
		\includegraphics[width=0.80\textwidth]{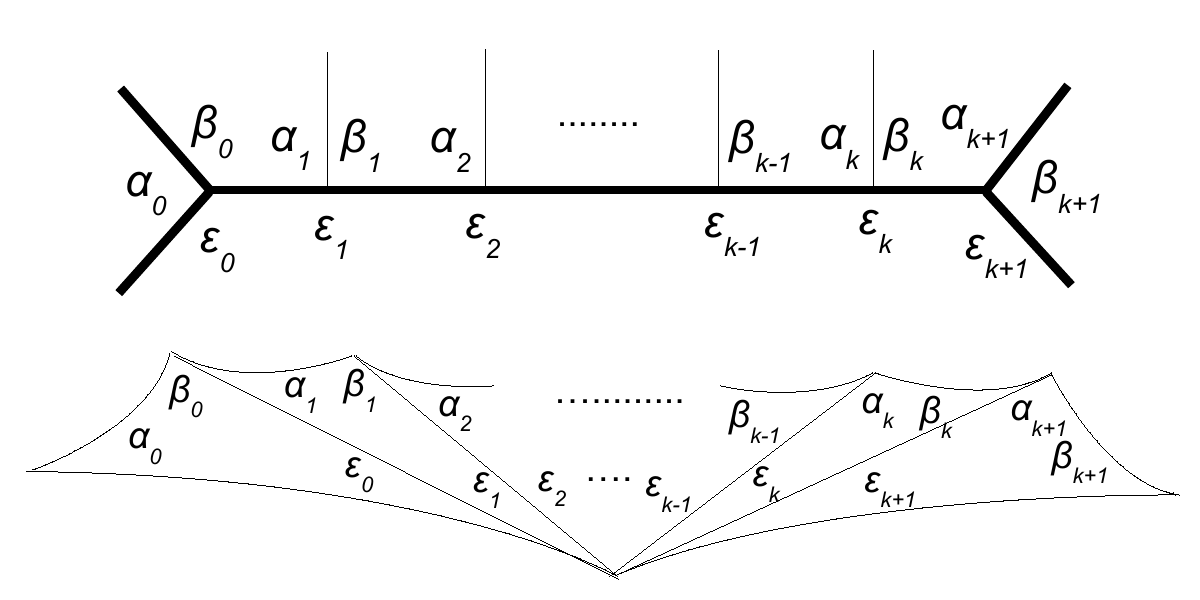}
	\caption{{\small Edges with bounded simplicial coordinates incident on a subfatgraph whose lambda lengths diverge; both the fatgraph version and the dual q.c.d. version are shown.}}
	\label{fig:SimpCoordsB}
\end{figure}

The procedure for operation (c)  is justified by the following calculation:
Figure \ref{fig:SimpCoordsB} illustrates the included h-lengths near a consecutive sequence of edges with bounded simplicial coordinates adjacent to a linear chain of  edges in bold strokes with vanishing simplicial coordinates.  Using the expression for simplicial coordinates as a linear combination of included h-lengths, one finds that the sum of simplicial coordinates along the linear chain is given by
\[
\sum_{i=1}^{k+1} X(e_i) = (\beta_0 + \epsilon_0 - \alpha_0) + 2\sum_{i=1}^k \epsilon_i + (\alpha_{k+1} + \epsilon_{k+1} - \beta_{k+1}).
\]
The term $2\sum_{i=1}^k \epsilon_i$ goes to zero, thus leaving the expression for the usual simplicial coordinate of the single edge  arising from the linear chain in keeping with the specification of simplicial coordinates above; a similar calculation holds if the linear chain terminates at a punctured vertex.

\begin{example}
\label{ex:partpair}
The top of Figure \ref{fig:PartialPairingExampleB} illustrates the fatgraph $G(E)$ treated in Example \ref{ex:filtscreen} on which we consider the two filtered screens
\[
\vec{\cE} = (E-[(a-g)\cup(i-k)],(a-g),(i-k)),
\]
\[
\vec{\cE}^{\prime}= (E-[(a-g)\cup(i-k)],(i-k),(a-g))
\]
in the earlier notation.  Assuming that $i-k$ is a horocycle, the common punctured fatgraph $\bar G(\vec{\cE})=\bar G(\vec{\cE}')$
with partial pairing is illustrated on the bottom in 
Figure \ref{fig:PartialPairingExampleB}, where the deprojectivized simplicial coordinates are indicated by capital Roman letters.

\begin{figure}[htbp]
	\centering
		\includegraphics[width=0.60\textwidth]{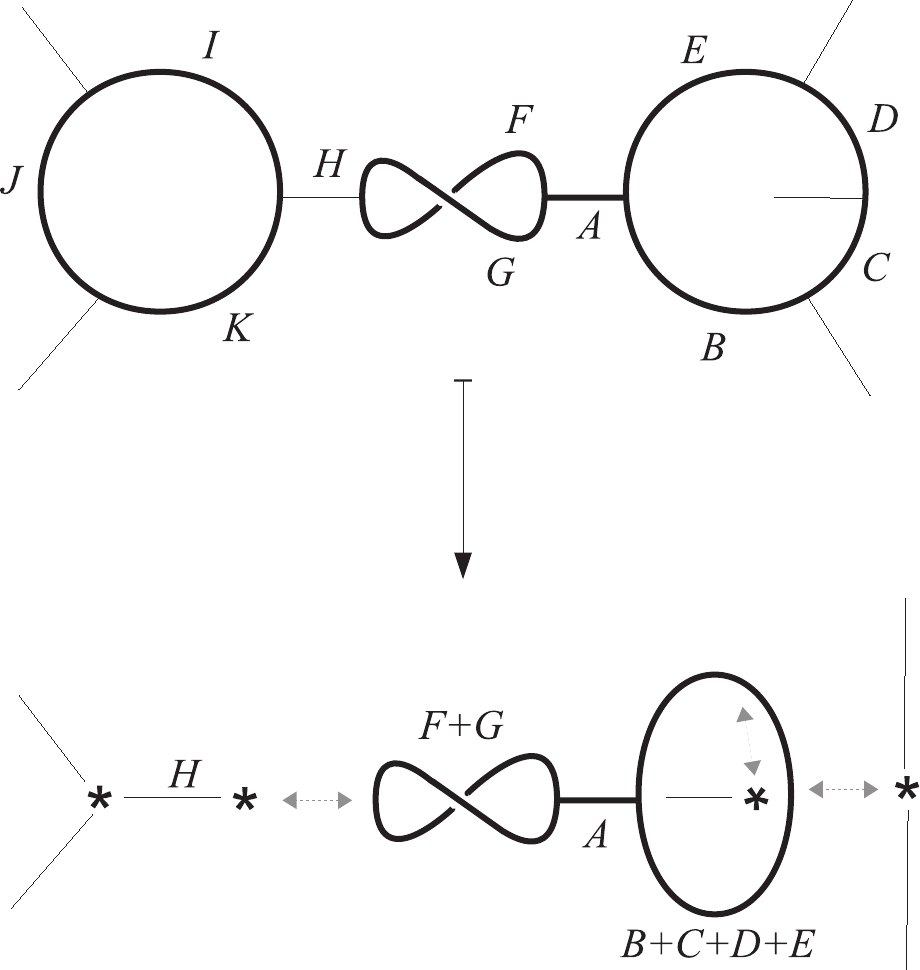}
	\caption{{\small Example of Construction \ref{constructit}.}}
	\label{fig:PartialPairingExampleB}
\end{figure}

\end{example}

\begin{definition} 
\label{defn:equivreln} Given $p\in\cC(\vec{\cE})\subset \cF\cS_{\cP}(F)$, 
the isotopy class of punctured fatgraph $\bar G(\vec{\cE})$ with partial pairing together with the
labeled assignment of projective simplicial coordinates on each component
punctured fatgraph determined by Construction \ref{constructit} is denoted $\pi (p)$.  Two points $p,q\in
\cF\cS_{\cP}(F)$ are {\it equivalent} and we write $p\sim q$ if $\pi (p)=\pi (q)$ with quotient space $\cF\cS_{\cP}(F)/\sim$.
\end{definition}

\begin{definition}[$\widehat T_{\cP}(F)$]
Define the {\em $\cP$-decorated augmented Teichm\"uller space} for $F$, denoted by $\widehat T_{\cP}(F)$, by the following:  a labeled isotopy class of punctured fatgraph with partial pairing $\bar G $ lies in $ \widehat T_{\cP}(F)$ if and only if $\bar G = \pi(p)$ for some point $p \in \cC(\vec{\cE}) \subset \cF\cS_{\cP}(F)$.  There is thus a bijection between $\widehat T_{\cP}(F)$ and $\cF\cS_{\cP}(F)/\sim$, and we topologize $\widehat T_{\cP}(F)$ with the quotient topology via this bijection.  
In the case of $\cF\cS(F)/\sim$ where $\cP$ is the set of all punctures of $F$, we simply denote this space by $\widehat T(F)$, and refer to it as the {\em decorated augmented Teich\"uller space} for $F$.   
\end{definition}

The example of $\widehat T(F_0^3)$ is illustrated in Figure \ref{fig:PGExampleB}. 
\begin{figure}[htbp]
	\centering
		\includegraphics[width=0.75\textwidth]{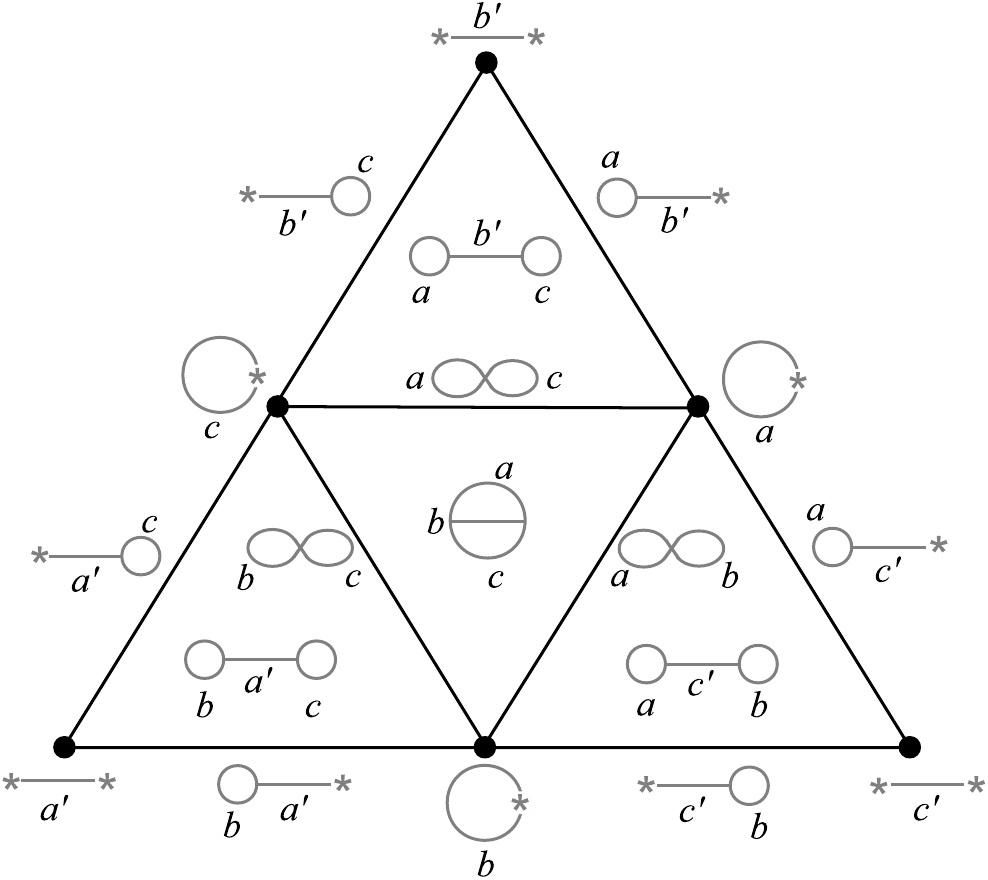}
	\caption{{\small $\widehat T(F_0^3)$ and its cell decomposition.}}
	\label{fig:PGExampleB}
\end{figure}

Our final goal in this paper is to provide a more usable description of a CW decomposition for $\widehat T(F)$, in particular describing how the strata of $\widehat T(F)$ are products of projectivized decorated Teichm\"uller spaces that are indexed by a new object termed {\it partially oriented stratum graphs}.

%------------------------------------------------------------------------------
\subsection{Partially oriented stratum graphs}
\label{subsec:nestsimbalanced}

Strata in $\bar T(F)$ are in 1-1 correspondence with stratum graphs $\cG(\sigma)$; we now begin to work toward showing that strata in $\widehat T(F)$ are in 1-1 correspondence with a new object, defined below, termed ``partially oriented stratum graphs''.

\begin{definition}[Partially oriented stratum graph]  A {\em partial orientation} on a stratum graph is an assignment of orientations to a subset of its edges such that every $\bullet$ vertex is an initial point for an oriented edge, no oriented edge has its initial point at a $\circ$ vertex, and there are no oriented cycles.
\end{definition}

We will typically denote a partially oriented stratum graph by the notation $\vec{\cG}$.

The following will be useful:  fix some stratum graph $\cG(\sigma)$ with edges $N \cup P$.  We shall consider ordered partitions $f$ of $N \cup P$, meaning $f$ is an ordered $m+1$-tuple $$f = (L_f^0,L_f^1,\cdots,L_f^m)$$ with $N \cup P = \bigcup_{i=1}^m L_f^i$, $L_f^i \neq \emptyset$ and $L_f^i \cap L_f^j = \emptyset$ for all $0 \leq i\neq j \leq m$.  

We use the notation $f$ since such an ordered partition exists if and only if it induces a surjection $$f:N\cup P \to \{0,1,\cdots,m\}$$ defined by $f(x) = k$ if $x \in L_f^k$; we will refer to $f(x)$ as the {\em level} of $x \in N \cup P$.

\begin{definition}[Nest]
\label{defn:nest}
An ordered partition $f$ is  a {\em nest} if for every $n \in N$ there is an adjacent $x \in N \cup P$ with $f(x) < f(n)$, where here $x$ is adjacent to $n$ if they share a common vertex endpoint.
\end{definition}

%\subsection{Partially oriented stratum graphs and compatible nests}

  We next observe that given $\vec{\cG}$, certain nests on its underlying stratum graph $\cG$ satisfy useful relationships with the orientations assigned to $\vec{\cG}$.

\begin{definition}[Compatible nest]
\label{defn:nestcompatible}
Given a partially oriented stratum graph $\vec{\cG}$, we say that a nest $f$ on the underlying stratum graph $\cG$ is {\em compatible with} $\vec{\cG}$ if for every vertex $v \in \cG$, the set of oriented edges with initial point at $v$ is precisely $\textrm{Min}_f(v)$. 
\end{definition}

We now observe that for each $\vec{\cG}$ such compatible nests exist.

\begin{definition}[The nest $f_{\vec{\cG}}$]
\label{defn:fbarg}
Given a partially oriented stratum graph $\vec{\cG}$, for each vertex $v$ we define $d(v)$ to be the maximal count of $\bullet$ vertices traversed by an oriented path starting from $v$ without counting $v$ itself; we shall call such a path realizing this maximal count a {\em maximal path} for $v$.  Define $f_{\vec{\cG}}:N \cup P \rightarrow \{0\} \cup \mathbb{N}$ as follows:  If $x \in N \cup P$ is an oriented edge, define $f_{\vec{\cG}}(x) = d(v)$ where $v$ is the initial vertex for the oriented $x$.  If $x \in N \cup P$ is unoriented, define $f_{\vec{\cG}}(x)$  to be one greater than $\max\{d(v) | v \in V(x)\}$.

\end{definition}

\begin{lemma}
\label{lemma:canonicalnest2}
The function $f_{\vec{\cG}}$ is a nest on $\vec{\cG}$.
\end{lemma}

\begin{proof}
To see that the nest condition of Definition \ref{defn:nest} is satisfied, first observe that for an unoriented $n \in N$, there are by definition adjacent $x$ for which $f_{\vec{\cG}}(x) < f_{\vec{\cG}}(n)$.  If $n$ is oriented, assume it is in $\textrm{Min}_{f_{\vec{\cG}}}(v)$ for some $v \in V(n)$ and let $v_0$ denote the other vertex of $n$.  Consider directed paths beginning at $v_0$ and observe that a maximal path beginning at $v_0$ must have length smaller than the maximal path beginning at $v$. 
Finally, to confirm that $f_{\vec{\cG}}$ is a surjection onto $\{0,1,\cdots,m\}$ for some $m$, consider an oriented edge of level $k$.  Its initial point is a vertex $v$ with $d(v)=k$, and the maximal path for $v$ has length $k$.  Along this path, the maximal path length for the next vertex after $v$ must be $k-1$.  Thus, there must be an edge of level $k-1$ in $f_{\vec{\cG}}$.  A similar statement holds by definition for unoriented edges. 
\end{proof}

Figure \ref{fig:PartOrientStratumGraph} shows an example of a partially oriented stratum graph $\vec{\cG}$ along with the compatible nest $f_{\vec{\cG}}$.

\begin{figure}[htbp]
	\centering
		\includegraphics[width=0.65\textwidth]{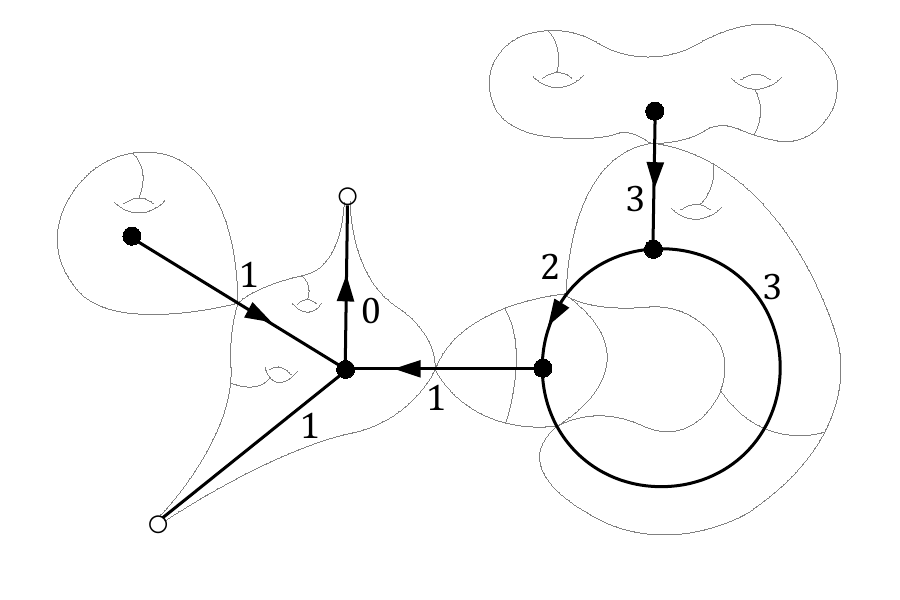}
	\caption{{\small An example of a partially oriented stratum graph $\vec{\cG}$ along with the compatible nest $f_{\vec{\cG}}$.}}
	\label{fig:PartOrientStratumGraph}
\end{figure}

\begin{definition}[The partially oriented stratum graph $\vec{\cG}_{\bar G}$]
\label{defn:thepartoriented}
To every point $\bar G \in \widehat T(F)$ we associate a unique stratum graph $\cG_{\bar G}$, which has a $\bullet$ vertex for each component punctured fatgraph and an edge for each pair of paired punctures  or puncture paired with a  boundary component connecting the vertices corresponding to the punctured fatgraphs (which may coincide) on its two sides; for each unpaired puncture (which includes both unpaired punctured vertices and unpaired boundary components), we add yet another edge connecting a univalent $\circ$-vertex to the vertex of the associated punctured fatgraph.  For every such stratum graph $\cG_{\bar G}$, we uniquely orient a subset of its edges to obtain a {\em partially oriented stratum graph} $\vec{\cG}_{\bar G}$, as follows:  for edges corresponding to a pairing of a decorated with an undecorated puncture, we orient from the decorated towards the undecorated puncture; for edges corresponding to unpaired decorated punctures, we orient towards the unpaired $\circ$-vertex and all other edges are left unoriented.
\end{definition}

In other words, given $\cG_{\bar G}$, we orient edges in the direction of decreasing level in filtered screens in the preimage $\pi^{-1}(\bar G)$; as a result, there do not exist oriented cycles in $\vec{\cG}_{\bar G}$, and every $\bullet$ vertex is the initial point for an oriented edge, as every component fatgraph of $\bar G$ has a decorated puncture.  It is then clear that every point $\bar G \in \widehat T(F)$ is a point in a product of projectivized partially decorated Teichm\"uller spaces, where the punctures decorated at each irreducible component are determined by the outgoing oriented edges at the associated vertex in the partially oriented stratum graph $\vec{\cG}_{\bar G}$; see Figure \ref{fig:PartOrientStratumGraphDec} for an example.  Every partially oriented stratum graph $\vec{\cG}$ yields in this manner a product of projectivized partially decorated Teichm\"uller spaces, and this leads us to the final lemma in this section.

\begin{figure}[htbp]
	\centering
		\includegraphics[width=0.65\textwidth]{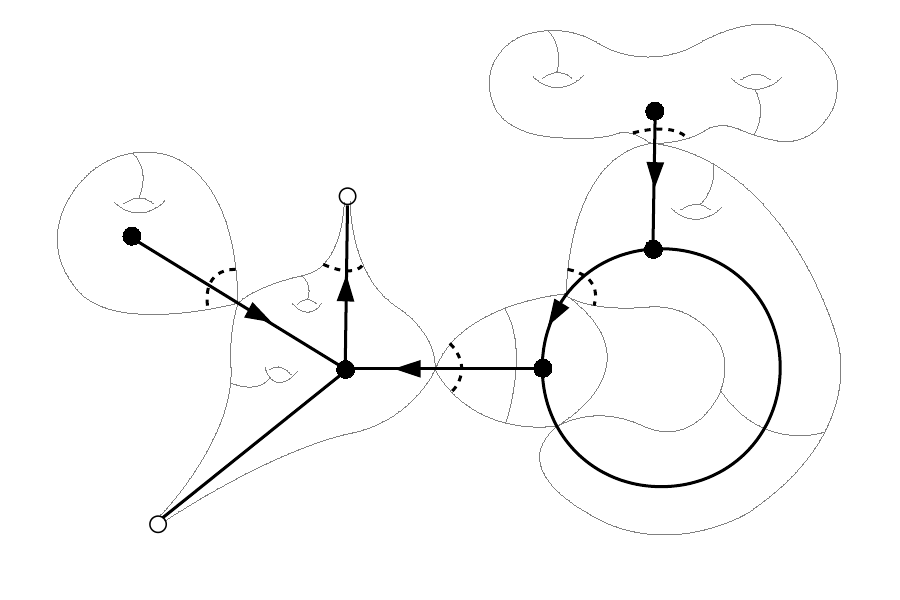}
	\caption{{\small An example of a partially oriented stratum graph $\vec{\cG}$ along with the decorated punctures at each irreducible component indicated by dashed arcs, and corresponding to outgoing edges.}}
	\label{fig:PartOrientStratumGraphDec}
\end{figure}

\begin{lemma}
\label{lem:prodpartproj}
Let $\vec{\cG}$ be a partially oriented stratum graph; then every point in the associated product of projectivized partially decorated Teichm\"uller spaces is in fact a point in $\widehat T(F)$. 
\end{lemma}

\begin{proof}
Fixing a point in the product of projectivized partially decorated Teichm\"uller spaces associated with $\vec{\cG}$ yields a punctured fatgraph with partial pairing, which we call $\bar G$.  Take $\bar{G}$, and resolve all punctured vertices and pairings, making choices, using the operations in Figure \ref{fig:ResolutionsB}, to obtain a fatgraph $G$ representing $F$.  Note that the partial orientation condition associated to $\vec{\cG}$ precludes paired decorated punctures since oriented edges yield a decoration only in the irreducible component for which that orientation is outgoing.  Thus, the resolutions in Figure \ref{fig:ResolutionsB} are exhaustive.

\begin{figure}[htbp]
	\centering
		\includegraphics[width=0.65\textwidth]{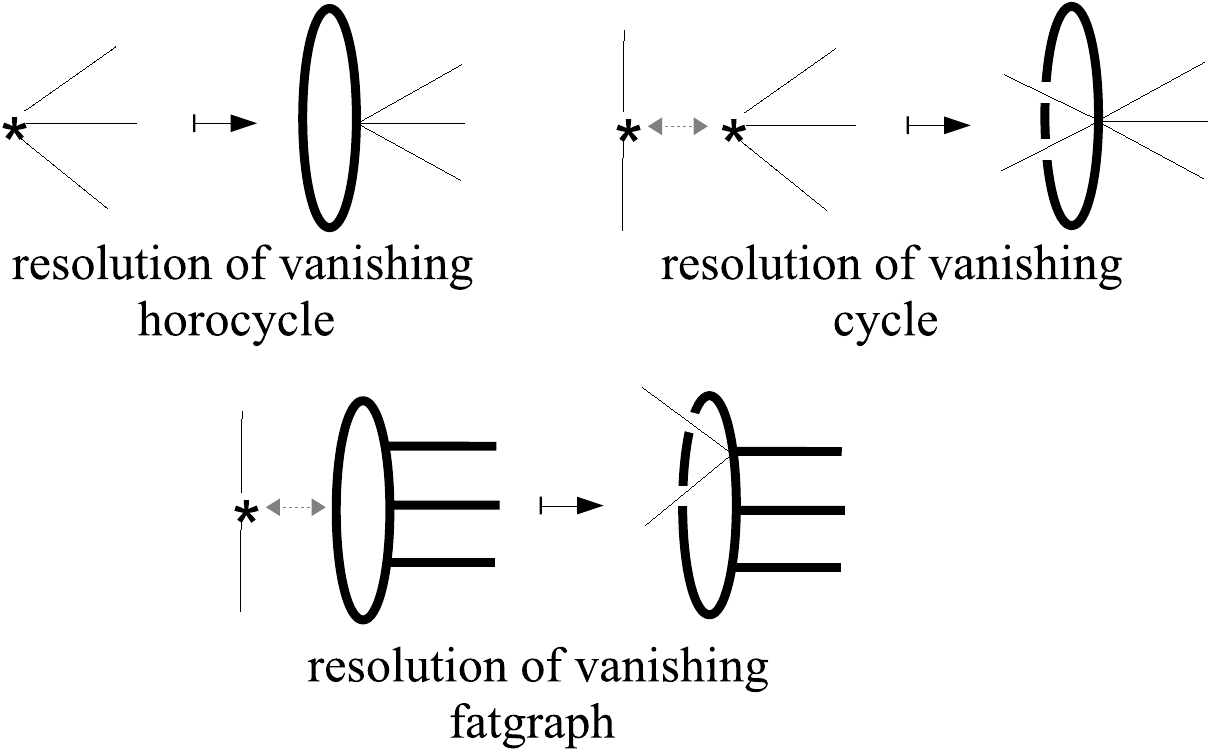}
	\caption{{\small Chosen resolutions of pairings and punctured vertices of $\bar{G}$ to obtain a fatgraph $G$.}}
	\label{fig:ResolutionsB}
\end{figure}

This chosen resolution produces a fatgraph $G$ for $F$. To determine the level structure for the filtered screen on $G$, choose a nest $f$ compatible with $\vec{\cG}$, and do the following:  For each edge $e$ in $G$ which is also in $\bar{G}$, the level of $e$ will be the $f$-level of edges in $\textrm{Min}_f(v_i)$, where $v_i$ is the $\bullet$ vertex representing the irreducible component $F_i$ containing $e \in \bar G$.  Furthermore, for edges of $G$ in horocycles or cycles resulting from the resolution of $\bar{G}$, assign the $f$-level associated to the corresponding edge in $\vec{\cG}$.  The claim is that this level structure yields a filtered screen, and if the maximum level is $n$, it suffices to show that $L^{\geq k}$ is recurrent for all $k = 1,\cdots,n$. 
To see this, fix $k$ and consider a component of $L^k$; if this  
component is a horocycle, cycle, or fatgraph, then it is recurrent.   
Otherwise, the edges in the component of $L^k$ appear in $\bar G$ as a  
punctured fatgraph with punctured vertices, corresponding to an  
irreducible component $F_i$.  However in $G$, these punctured vertices  
no longer appear, and edges incident on such punctured vertices of  
$\bar G$ are now incident on a vertex of $G$ in either a vanishing  
horocycle, cycle, or fatgraph, all of which are of greater level, and  
hence in $L^{\geq k}$.
Thus, $L^{\geq k}$ is recurrent, and the level structure yields a filtered screen $\vec{\cE}$.
  
We need to determine correct coordinates for a point $p \in \cC(\vec{\cE})$ so that $\pi(p)=\bar G$.  The  deprojectivized coordinates for any horocycle or cycle in $G$ resulting from a resolution can simply be assigned arbitrarily.  For all other edges, the deprojectivized coordinates are identical to those in $\bar G$; now projectivize in comparability classes in the level structure of $\vec{\cE}$ to obtain $p \in \cC(\vec{\cE})$ for which $\pi(p)=\bar G$.
\end{proof}

%-----------------------------------------------------------------------
\subsection{CW decomposition and strata of $\widehat T(F)$}
\label{subsec:celldecompstrata}

We now come to the last of our main results, which is an $MC(F)$-invariant CW decomposition of $\widehat T(F)$ which follows directly from the above analysis.

\begin{theorem}
\label{thm:maincelldecomp}
$\widehat T(F)$ admits an $MC(F)$-invariant CW decomposition with the following properties:

\begin{itemize}
	\item Cells are in one-to-one correspondence with isotopy classes of punctured fatgraphs with partial pairing $\bar G$ supported by $F$ for which $\vec{\cG}_{\bar G}$ is a partially oriented stratum graph.
	\item Each cell is a product of simplices, where projectivized simplicial coordinates on the component punctured fatgraphs give barycentric coordinates for the factor simplices.
	\item The union of cells with the same underlying partially oriented stratum graph is itself a product of projectivized partially decorated Teichm\"{u}ller spaces, one for each irreducible component; we shall refer to these as the {\em strata} of $\widehat T(F)$.
	\item The face relation for cells within a stratum is generated by collapsing edges (of punctured fatgraph components) with distinct endpoints, at most one of which is punctured.
	\item The face relation for cells in different strata is generated by vanishing of quasi recurrent subsets within an irreducible component $F_i$, followed by the map $\pi$ applied to the resulting filtered screen in $\cF\cS(F_i)$.
\end{itemize}
\end{theorem}

\begin{proof}
We see from Lemma \ref{lem:prodpartproj} that $\widehat T(F)$ is composed of strata of products of projectivized partially decorated Teichm\"{u}ller spaces, where a particular stratum has an underlying stratum $T_\sigma$ from augmented Teichm\"{u}ller space but with a choice of decorated subset $\cP_i$ of the punctures on each irreducible component $F_i$.  We may label these strata by $\widetilde{T}_{\cup \cP_i}(F_\sigma)$, where the $F_\sigma$ indicates the underlying stratum in augmented Teichm\"{u}ller space, and the $\cup \cP_i$ indicates the decorated punctures on each irreducible component $F_i$.  

However, the subset of punctures which are decorated cannot be chosen arbitrarily but must be chosen consistently with a partially oriented stratum graph $\vec{\cG}$ that supports a compatible nest.  Nevertheless, as long as these conditions are satisfied, each stratum admits a cell decomposition, where cells are in one-to-one correspondence with isotopy classes of punctured fatgraphs with partial pairing compatible with the underlying partially oriented stratum graph.  Each cell is a product of simplices, where simplicial coordinates on the component punctured fatgraphs give barycentric coordinates for the factor simplices.  The face relation for cells within a stratum is generated by collapsing edges with distinct endpoints (only one of which may be punctured) in the punctured fatgraphs associated to irreducible components.    

Moreover, since the face relation on $\cF\cS(F)$ is also generated by coalescing of adjacent levels in filtered screens using the fact that $\widehat T(F)$ is then obtained by a quotient of $\cF\cS(F)$ via the map $\pi$, we see that the face relation for cells in different strata in $\widehat T(F)$ is generated by way of $\cF\cS_{\cP_i}(F_i)$ and $\pi$ for each irreducible component $F_i$ in $F_\sigma$.  More specifically, given a cell $\cC(E_1) \times \cdots \times \cC(E_m)$ in a stratum $\widetilde{T}_{\cup \cP_i}(F_\sigma)$, where $F_\sigma$ has irreducible components $F_1,\cdots,F_m$, we consider a stable path $\bar X_t$ which only varies simplicial coordinates within one factor $\cC(E_j)$ and whose limiting screen $\vec{\cE}(\bar X_t)$ is of total level one.  There will then be a limiting point $p(\bar X_t) \in \cC(\vec{\cE}(\bar X_t)) \subset \cF\cS_{\cP_i}(F_i)$ to which we may apply $\pi$ to obtain a new labeled punctured fatgraph with partial pairing; this will yield a point in a new cell in a new stratum where the $\cC(E_j)$ factor has been replaced by a  product $\cC(E_{j_1}) \times \cdots \times \cC(E_{j_k})$.  If the stable path simply yielded a new vanishing horocycle, we have moved from $\widetilde{T}_{\cup \cP_i}(F_\sigma)$ to $\widetilde{T}_{\cup \cP^\prime_i}(F_\sigma)$, where the underlying stratum in augmented Teichm\"{u}ller space has stayed fixed, but we have changed the collection of punctures which are decorated and hence have changed the underlying partially oriented stratum graph without changing the stratum graph.  However, if the stable path yields a vanishing cycle or fatgraph, then we move to $\widetilde{T}_{\cup \cP^\prime_i}(F_{\sigma^\prime})$, where the underlying stratum in augmented Teichm\"{u}ller space itself has changed.

Finally, we observe that the natural action of $MC(F)$ on $\widehat T(F)$ takes cells to cells, whence the decomposition is $MC(F)$-invariant.
\end{proof}

\end{document}